\documentclass[11pt]{article}
\usepackage[a4paper]{geometry}
\usepackage[utf8]{inputenc}
\usepackage{graphicx}
\usepackage{dsfont}
\usepackage{fancyhdr}
\usepackage{booktabs}
\usepackage[colorlinks=true, allcolors=blue,pagebackref=true]{hyperref}
\usepackage{url}
\usepackage{xcolor}
\usepackage{colortbl}
\usepackage{subfig}
\usepackage{mathtools}
\usepackage[T1]{fontenc}
\usepackage{comment}
\usepackage{natbib}
\usepackage{eurosym}
\usepackage{babel}
\usepackage{amsmath}
\allowdisplaybreaks
\usepackage{xcolor}
\usepackage{enumitem}
\usepackage{float}
\usepackage{here}
\usepackage{wrapfig}
\usepackage{lscape}
\usepackage{bbold}
\usepackage{stmaryrd}
\usepackage[export]{adjustbox}
\usepackage{tabularx}
\usepackage{amssymb}
\usepackage{mathtools}
\usepackage{multirow,makecell, framed}
\usepackage[onelanguage, boxruled, vlined]{algorithm2e}
\SetAlFnt{\footnotesize}
\usepackage{amsthm, bm}
\usepackage{lscape}
\usepackage{ulem}
\usepackage{vmargin}
\usepackage{siunitx}

\newcommand{\mydot}{\mathrel{\scalebox{0.6}{$\bullet$}}}

\newcommand{\dint}{\displaystyle\int}
\newcommand{\dsum}{\displaystyle\sum}
\newcommand{\dprod}{\displaystyle\prod}

\newtheorem{lemma}{Lemma}
\newtheorem{proposition}{Proposition}
\newtheorem{theorem}{Theorem}
\newtheorem{definition}{Definition}

\theoremstyle{remark}
\newtheorem{remark}{Remark}

\usepackage{authblk}

\renewcommand*{\backref}[1]{}
\renewcommand*{\backrefalt}[4]{}

\setmarginsrb{2cm}{1cm}{2cm}{1cm}{0.5cm}{0.5cm}{0.5cm}{0.5cm}

\title{A multivariate spatial regression model using signatures}
\author[,1]{Camille Frévent\thanks{Corresponding author: \texttt{camille.frevent@univ-lille.fr}}}
\affil[1]{Univ. Lille, CHU Lille, ULR 2694 - METRICS: Évaluation des technologies de santé et des pratiques médicales, F-59000 Lille, France}
\author[2,3]{Issa-Mbenard Dabo}
\affil[2]{Institut de mathématiques de Bordeaux, University of Bordeaux, France}
\affil[3]{Departement of Mathematics, Division of Science, New York University, Abu Dhabi, UAE}

\date{}

\begin{document}

\maketitle

\hrule
\section*{Abstract}
We propose a spatial autoregressive model for a multivariate response variable and functional covariates. The approach is based on the notion of signature, which represents a function as an infinite series of its iterated integrals and presents the advantage of being applicable to a wide range of processes. \\
We have provided theoretical guarantees for the choice of the signature truncation order, and we have shown in a simulation study and an application to pollution data that this approach outperforms existing approaches in the literature. \\[0.2cm]
\textbf{Keywords:} Functional data, Multivariate regression, Signature, Spatial regression, Tensor \vspace{0.5cm}
\hrule

\section{Introduction}

Advances in sensing technology and data storage capacities have led to an increasing amount of continuously recorded data over time. This led to the introduction of Functional Data Analysis (FDA) by \cite{ramsaylivre} and to the adaptation of numerous statistical approaches to the functional framework. We are interested here in regression models for a real response variable and functional covariates observed over a time interval $\mathcal{T}$. In this context, the traditional approach assumes that the functional covariate $X$ belongs to $\mathcal{L}^2(\mathcal{T}, \mathbb{R}^P)$, the space of $P$-dimensional square-integrable functions on $\mathcal{T}$, and considers the following model \citep{ramsaylivre}:
$$ Y = \dint_{\mathcal{T}} X(t)^\top \beta^*(t)  \ \text{d}t + \varepsilon. $$ 
This model is usually estimated by approximating $X$ and $\beta$ as finite combinations of basis functions and then using classical linear regression estimation on the obtained coefficients. \\

In recent years, signatures - initially defined by \cite{chen1957integration, chen1977iterated} for smooth paths and rediscovered in the context of rough path theory \citep{lyons1998differential, friz2010multidimensional} - have gained popularity in many fields such as character recognition \citep{graham2013sparse,liu2017ps,7995142}, medicine \citep{perez2018signature, morrill2020utilization} and finance \citep{gyurko2013extracting, arribas2018derivatives, perez2020signatures}. \cite{fermanian2022functional} first proposed a linear regression model for a real response variable and functional covariates using their signatures, highlighting three main advantages: (i) signatures do not require $X \in \mathcal{L}^2(\mathcal{T}, \mathbb{R}^P)$, (ii) they are naturally adapted to multivariate functions, and (iii) they encode the geometric properties of $X$. \\

In domains where data inherently involve a spatial component (e.g., environmental science), functional data analysis has led to the development of methods specifically designed for spatial functional data. In the context of spatial regression, \cite{huang2018spatial} and \cite{ahmed2022quasi} assumed $X \in \mathcal{L}^2(\mathcal{T}, \mathbb{R})$ and proposed the following functional spatial autoregressive model (FSARLM):
$$ Y_i = \rho^* \dsum_{j=1}^n W_{i,j} Y_j + \dint_\mathcal{T} \beta(t)^* X(t) \ \text{d}t + \varepsilon_i $$
where $W = (W_{i,j})_{1 \le i,j \le n}$ is a non-stochastic spatial weights matrix and $\rho^*$ is a spatial autoregressive parameter in $[-1,1]$. \\
Following the popularization of signatures, \cite{frevent2023functional} introduced two spatial regression models for functional covariates based on a SAR model and signatures: the ProjSSAR and the PenSSAR. Briefly, the ProjSSAR is based on a Principal Component Analysis (PCA) applied to signatures and a spatial regression estimation using the PCA scores, while the PenSSAR employs a penalized spatial regression. \\

In the context of a multivariate response variable, \cite{yang2017identification} and \cite{zhu2020multivariate} proposed spatial regression models for non-functional covariates. However, to our knowledge, no spatial regression model has been developed for a multivariate response variable and functional covariates. This motivated us to develop a new model in this context, combining the MSAR proposed by \cite{zhu2020multivariate}, the PenSSAR \citep{frevent2023functional}, and Ridge penalization \citep{yanagihara2010unbiased}. \\

Section \ref{sec:sig} introduces the signatures and their properties. Section \ref{sec:model} presents the proposed multivariate penalized signatures-based spatial regression model as well as its estimation procedure and theoretical guaranties. Section \ref{sec:simu} describes a simulation study comparing the new model to the FSARLM \citep{ahmed2022quasi}, the PenSSAR and the ProjSSAR \citep{frevent2023functional}. Our method is then applied to a real dataset in Section \ref{sec:appli}. Finally, Section \ref{sec:discussion} concludes the paper with a discussion.

\section{Signature of a path} \label{sec:sig}

We provide in this section a brief presentation of signatures and we refer the reader to \cite{lyons2007differential,friz2010multidimensional} for a more complete description.
The signature of a smooth path $\mathcal{X}$ is an infinite sequence of tensors defined by iterated integrals that gathers information about $\mathcal{X}$. We assume the covariate $\mathcal{X}: \mathcal{T} \rightarrow \mathbb{R}^P$ to be a continuous path of bounded variation, that is it is continuous and
$$
\left\| \mathcal{X}\right\|_{TV} = \sup_{(t_0,...,t_k)\in \mathcal{I}} \sum_{i=1}^k \left\| \mathcal{X}_{t_i} - \mathcal{X}_{t_{i-1}}\right\| < +\infty,
$$
where $\left\| .\right\|_{TV}$ denotes the total variation distance, $\left\| .\right\|$ denotes the Euclidean norm on $\mathbb{R}^P$ and $\mathcal{I}$ denotes the set of all finite partitions of $\mathcal{T}$. We denote by $\mathcal{C}(\mathcal{T},\mathbb{R}^P)$, the set of path of bounded variation on $\mathcal{T}$.

We are now able to define the signature of a continuous path of bounded variation.

\begin{definition}
Let $\mathcal{X} \in \mathcal{C}(\mathcal{T},\mathbb{R}^P)$, the signature of $\mathcal{X}$ is the following sequence 
$$ Sig(\mathcal{X}) = (1, \bm{\mathcal{X}}^1, \dots, \bm{\mathcal{X}}^k, \dots) $$
where $$ \bm{\mathcal{X}}^k = \underset{\substack{t_1 < \dots < t_k \\ t_1, \dots, t_k \in \mathcal{T}}}{\dint \dots \dint} \ \text{d}\mathcal{X}_{t_1} \otimes \dots \otimes \text{d}\mathcal{X}_{t_k} \in \left(\mathbb{R}^{P}\right)^{\otimes k}. $$
\end{definition}

\begin{definition}
Let $\mathcal{X} \in \mathcal{C}(\mathcal{T},\mathbb{R}^P)$. We define the signature coefficients vector of $\mathcal{X}$ as the following sequence 
$$
\widetilde{S}(\mathcal{X})=\left(1,S^{(1)}(\mathcal{X}), \dots, S^{(P)}(\mathcal{X}), S^{(1,1)}(\mathcal{X}), S^{(1,2)}(\mathcal{X}), \dots, S^{\left(i_1, \dots, i_k \right)}(\mathcal{X}), \dots \right),
$$
and the shifted-signature coefficients vector of $\mathcal{X}$ is defined as follows
$$
S(\mathcal{X})=\left(S^{(1)}(\mathcal{X}), \dots, S^{(P)}(\mathcal{X}), S^{(1,1)}(\mathcal{X}), S^{(1,2)}(\mathcal{X}), \dots, S^{\left(i_1, \dots, i_k\right)}(\mathcal{X}), \dots\right),
$$
where for all $k \ge 1$ and for all multi-index , $I=(i_1, \dots, i_k) \subset\{1, \dots, P\}^k$ of length $k$, $S^I(\mathcal{X})$ is the signature coefficient of order $k$ along $I$ on $\mathcal{T}$ defined as the following iterated integral:
$$
S^I(\mathcal{X})=\underset{\substack{t_1 < \dots < t_k \\ t_1, \dots, t_k \in \mathcal{T}}}{\dint \dots \dint} \ \text{d}\mathcal{X}_{t_1}^{(i_1)} \dots \text{d}\mathcal{X}_{t_k}^{(i_k)}.
$$
\end{definition}

A signature coefficients vector is an infinite sequence of iterated integrals, however it is more convenient to use finite sequences. Therefore we define in the following the truncated signature.

\begin{definition}
Let $\mathcal{X} \in \mathcal{C}(\mathcal{T},\mathbb{R}^P)$ and $m \ge 0$. The truncated signature coefficients vector of $\mathcal{X}$ at order $m$, denoted by $\widetilde{S}^m(\mathcal{X})$, is the sequence of signature coefficients of order $k \le m$, that is
$$
\widetilde{S}^m(\mathcal{X})=(1, S^{(1)}(\mathcal{X}), S^{(2)}(\mathcal{X}), \dots, S^{\overbrace{ \scriptstyle(P, \ldots, P)}^{\text {length m }}}(\mathcal{X})).
$$
We define similarly the truncated shifted-signature coefficients vector of $\mathcal{X}$:
$$
S^m(\mathcal{X})=(S^{(1)}(\mathcal{X}), S^{(2)}(\mathcal{X}), \dots, S^{\overbrace{\scriptstyle(P, \ldots, P)}^{\text {length m }}}(\mathcal{X})).
$$
\end{definition}

The truncated shifted-signature coefficients vector is then a vector of length $s_P(m)$, where 
$$
s_P(m) = \sum_{k = 1}^m P^k = \frac{P^{m+1}-P}{P-1} \quad \mbox{ for } P\geq 2 \mbox{ and } s_1(m) = m.
$$

\begin{theorem}[Proposition 2 from \cite{fermanian2022functional}] \label{th:approx} \phantom{linebreak} \\
Let $f: \mathcal{D} \rightarrow \mathbb{R}$ be a continuous function where $\mathcal{D} \subset \mathcal{C}(\mathcal{T},\mathbb{R}^P)$ is a compact subset such that for any $\mathcal{X}\in \mathcal{D}$, $\mathcal{X}_0 = 0$. \\
Let $\mathcal{X}\in \mathcal{D}$, we define $\widetilde{\mathcal{X}}=\left(\mathcal{X}_t^\top, t\right)_{t \in \mathcal{T}}^{\top}$ the associated time-augmented path. \\
Then, for every $\delta>0$, there exists $m^* \in \mathbb{N}, \beta^*_{m^*} \in \mathbb{R}^{s_P\left(m^*\right)+1}$, such that, for any $\mathcal{X} \in \mathcal{D}$,
$$
\left|f(\mathcal{X})-\left\langle\beta^*_{m^*}, \widetilde{S}^{m^*}(\widetilde{\mathcal{X}})\right\rangle\right| \le \delta,
$$
where $\langle\cdot, \cdot\rangle$ denotes the Euclidean scalar product on $\mathbb{R}^{s_P\left(m^*\right)+1}$.
\end{theorem}

In the next sections we adopt the more conventional notation for functional data
$$\begin{array}{lccl}
\mathcal{X}: & \mathcal{T} & \rightarrow & \mathbb{R}^P \\
& t & \rightarrow & \left( \mathcal{X}^{(1)}(t), \dots, \mathcal{X}^{(P)}(t) \right)^\top
\end{array} .$$

\section{The multivariate penalized signatures-based spatial regression \\ (MPenSSAR) model} \label{sec:model}

In the following sections we denote 
$\mathcal{M}_{s_P(m) \times Q}(\mathbb{R})$ the set of real matrices of size $s_P(m) \times Q$ and $\mathcal{M}_{Q}([-1,1])$ the set of real square matrices of size $Q \times Q$ with values in $[-1,1]$, $\mathcal{B}_{s_P(m) \times Q, \alpha}$ the ball composed by the real matrices of size $s_P(m) \times Q$ with a Frobenius norm less than $\alpha$, $\bm{0}_Q$ the column vector consisting of $Q$ times the value 0 and $\bm{1}_n$ the column vector consisting of $n$ times the value 1.

\subsection{The model}
We assume that $\mathcal{X}(0)=\bm{0}_P$ and that $\mathcal{X}$ has been time-augmented. Then, Theorem \ref{th:approx} motivates us to consider the following model for the process $\mathcal{Y} = \{ \mathcal{Y}(s_i) = \mathcal{Y}_i \in \mathbb{R}^Q, 1 \le i \le n \}$ in $n$ spatial units $s_1, \dots, s_n $:
\begin{equation} \label{eq:model}
\bm{\mathcal{Y}} = W\bm{\mathcal{Y}}R^* + \bm{1}_n \mu^*  + \bm{S^{m^*}(\mathcal{X})} \beta_{m^*}^* + \bm{\varepsilon}    
\end{equation}

with $\bm{\mathcal{Y}} = (\mathcal{Y}_1^\top, \dots, \mathcal{Y}_n^\top)^\top \in \mathcal{M}_{n \times Q}(\mathbb{R})$,
$\bm{S^{m^*}(\mathcal{X})} = (S^{m^*}(\mathcal{X}_1)^\top, \dots, S^{m^*}(\mathcal{X}_n)^\top)^\top \in \mathcal{M}_{n \times s_P(m^*)}(\mathbb{R})$ and $\bm{\varepsilon} = (\varepsilon_1^\top, \dots, \varepsilon_n^\top)^\top \in \mathcal{M}_{n \times Q}(\mathbb{R})$ where the disturbances $\{\varepsilon_i \in \mathbb{R}^Q, 1 \le i \le n\}$ are assumed to be independent and identically distributed random variables that are independent of $\{X_i(t) \in \mathbb{R}^P, t \in \mathcal{T}, 1 \le i \le n \}$ and such that $\mathbb{E}(\varepsilon_i)=\bm{0}_Q$ and $\mathbb{V}(\varepsilon_i)=\Sigma \in \mathcal{M}_Q(\mathbb{R})$ .

$\mu^* \in \mathcal{M}_{1 \times Q}(\mathbb{R})$, $\beta_{m^*}^* \in \mathcal{M}_{s_P(m^*) \times Q}(\mathbb{R})$, and $R^* \in \mathcal{M}_Q([-1,1])$ is such that its diagonal elements $\{R^*_{q,q}, 1 \le q \le Q\}$ represent the spatial effects of the $q^\text{th}$ variable in $\mathcal{Y}$ on itself and the elements outside its diagonal $\{R^*_{q,q'}, 1 \le q,q' \le Q, q \neq q'\}$ represent the cross-variable spatial effects \citep{yang2017identification}. 

Finally, $W \in \mathcal{M}_n(\mathbb{R})$ is a spatial weight matrix that it is common but not necessary to row normalize in practice. \\

Now we consider a sample of $\mathcal{Y}$ and $\mathcal{X}$ in the spatial locations $s_1, \dots, s_n$ ($\bm{Y}$ and $\bm{X}$), then:
$$ \bm{Y} = W\bm{Y}R^* + \bm{1}_n \mu^* + \bm{S^{m^*}(X)} \beta_{m^*}^* + \bm{e}$$
where $\bm{Y} = (Y_1^\top, \dots, Y_n^\top)^\top \in \mathcal{M}_{n \times Q}(\mathbb{R})$, $\bm{S^{m^*}(X)} = (S^{m^*}(X_1)^\top, \dots, S^{m^*}(X_n)^\top)^\top \in \mathcal{M}_{n \times s_P(m^*)}(\mathbb{R})$ and $\bm{e} = (e_1^\top, \dots, e_n^\top)^\top \in \mathcal{M}_{n \times Q}(\mathbb{R})$.

\subsection{Estimation}

In Model \ref{eq:model}, the parameters $R^*, \mu^*, \beta_{m^*}^*$ as well as the true truncation order $m^*$ are unknown and must be estimated. However, due to the large number $s_P(m^*) \times Q$ of coefficients in $\beta_{m^*}^*$ to be estimated, we need to use a penalized approach. In the following we will consider a Ridge regularization by assuming $(\mathcal{H}_\alpha): \exists \alpha >0 / \beta_{m^*}^* \in \mathcal{B}_{s_P(m^*) \times Q,\alpha} $. \\

Then, for a fixed truncation order $m$, we consider the objective function \citep{ma2020naive}:
$$ \mathcal{R}_m(\mu_m, \beta_m, R_m) = \mathbb{E}\left( \dfrac{1}{n} \left|\left| \bm{\mathcal{Y}} - W \bm{\mathcal{Y}} R_m - \bm{1}_n \mu_m - \bm{S^m(\mathcal{X})} \beta_m \right|\right|^2 \right), $$
which is minimal on $\mathcal{M}_{1 \times Q}(\mathbb{R}) \times \mathcal{B}_{s_P(m) \times Q, \alpha}(\mathbb{R}) \times \mathcal{M}_Q([-1,1])$ in 
$$ (\mu_m^*,\beta_m^*,R_m^*) = \underset{\substack{ \mu_m \in \mathcal{M}_{1 \times Q}(\mathbb{R}) \\
\beta_m \in \mathcal{B}_{s_P(m) \times Q, \alpha}(\mathbb{R}) \\
R_m \in \mathcal{M}_Q([-1,1])}}{\arg \min} \ \mathcal{R}_m(\mu_m, \beta_m, R_m). $$ 

Then, we denote
$$ L(m) = \mathbb{E}\left( \dfrac{1}{n} \left|\left| \bm{\mathcal{Y}} - W \bm{\mathcal{Y}} R_m^* - \bm{1}_n \mu_m^* - \bm{S^m(\mathcal{X})} \beta_m^* \right|\right|^2 \right). $$

These quantities can also be written on the sample $\bm{Y} = (Y_1^\top, \dots, Y_n^\top)^\top$ by considering the empirical objective function
$$\widehat{\mathcal{R}}_m(\mu_m,\beta_m,R_m) = \dfrac{1}{n} \left|\left| \bm{Y} - W\bm{Y}R_m - \bm{1}_n \mu_m - \bm{S^m(X)}\beta_m \right|\right|^2 $$
and its minimum on $\mathcal{M}_{1 \times Q}(\mathbb{R}) \times \mathcal{B}_{s_P(m) \times Q, \alpha}(\mathbb{R}) \times \mathcal{M}_Q([-1,1])$, which is reached in $\widehat{\mu}_m, \widehat{\beta}_m, \widehat{R}_m$:
$$\widehat{L}(m) = \widehat{\mathcal{R}}_m(\widehat{\mu}_m,\widehat{\beta}_m,\widehat{R}_m).$$

Now, it should be noted that minimizing
 
$$ \widehat{\mathcal{R}}_m(\mu_m,\beta_m,R_m) = \dfrac{1}{n} \left|\left| \bm{Y} - W \bm{Y} R_m -  \bm{\widetilde{S}^m(X)} \begin{pmatrix} \mu_m \\ \beta_m \end{pmatrix} \right|\right|^2 $$ on
$\mathcal{M}_{1 \times Q}(\mathbb{R}) \times \mathcal{B}_{s_P(m) \times Q, \alpha}(\mathbb{R}) \times \mathcal{M}_Q([-1,1])$
is equivalent to minimize 
$$ \widehat{\mathcal{R}}_m(\mu_m,\beta_m,R_m) + \lambda ||\beta_m||^2 = \dfrac{1}{n} \left|\left| \bm{Y} - W \bm{Y} R_m - \bm{\widetilde{S}^m(X)} \begin{pmatrix} \mu_m \\ \beta_m \end{pmatrix} \right|\right|^2 + \lambda ||\beta_m||^2 $$
on $\mathcal{M}_{1 \times Q}(\mathbb{R}) \times \mathcal{M}_{s_P(m) \times Q}(\mathbb{R}) \times \mathcal{M}_Q([-1,1])$,
where the Ridge parameter $\lambda$ depends on $\alpha$.

Thus,
\begin{align*}
(\widehat{\mu}_m, \widehat{\beta}_m, \widehat{R}_m) &= \underset{\substack{\mu_m \in \mathcal{M}_{1 \times Q}(\mathbb{R}) \\ \beta_m \in \mathcal{M}_{s_P(m) \times Q}(\mathbb{R}) \\ R_m \in \mathcal{M}_Q([-1,1])}}{\arg\min} \ \dfrac{1}{n} \left|\left| \bm{Y} - W \bm{Y} R_m - \bm{\widetilde{S}^m(X)} \begin{pmatrix} \mu_m \\ \beta_m \end{pmatrix} \right|\right|^2 + \lambda ||\beta_m||^2 \\
&= \underset{\substack{\mu_m \in \mathcal{M}_{1 \times Q}(\mathbb{R}) \\ \beta_m \in \mathcal{M}_{s_P(m) \times Q}(\mathbb{R}) \\ R_m \in \mathcal{M}_Q([-1,1])}}{\arg\min} \ \dfrac{1}{n} \dsum_{i=1}^n \left|\left| Y_i - W_{i,\mydot} \bm{Y} R_m - \widetilde{S}^m(X_i) \begin{pmatrix} \mu_m \\ \beta_m \end{pmatrix} \right|\right|^2 + \lambda ||\beta_m||^2.
\end{align*}

By deriving $$\dfrac{1}{n} \dsum_{i=1}^n \left|\left| Y_i - W_{i,\mydot} \bm{Y} R_m - \widetilde{S}^m(X_i) \begin{pmatrix} \mu_m \\ \beta_m \end{pmatrix} \right|\right|^2 + \lambda ||\beta_m||^2.$$ as a function of $\mu_m$ and $\beta_m$, we obtain the following estimators for these parameters as a function of $R_m$:
\begin{equation}
\label{eq:functionofRm}
\begin{pmatrix} \hat{\mu}_m(R_m) \\ \hat{\beta}_m(R_m) \end{pmatrix} = \left( \bm{\widetilde{S}^m(X)}^\top \bm{\widetilde{S}^m(X)} + n \Lambda \right)^{-1} \left( \bm{\widetilde{S}^m(X)}^\top (\bm{Y} - W \bm{Y} R_m) \right)
\end{equation}
where $\Lambda = \begin{pmatrix}
    0 & 0 & 0 & \dots & 0 \\
    0 & \lambda & 0 & \dots & 0 \\
    0 & 0 & \ddots & \ddots & \vdots \\
    \vdots & \ddots & \ddots & \ddots & 0 \\
    0 & \dots & 0 & 0 & \lambda
\end{pmatrix}$.

Then, we propose the following algorithm for a fixed truncation order $m$ and a fixed regularization matrix $\Lambda$:

\begin{footnotesize}
\SetKwComment{Comment}{/* }{ */}
\begin{algorithm}
\caption{Algorithm to estimate the MPenSSAR}
\KwData{$W, \bm{Y}, \bm{S^*(\mathcal{X})}, \Lambda$}
\KwResult{$\hat{R}_m, \hat{\mu}_m, \hat{\beta}_m$} \vspace{0.3cm}

$\begin{aligned}
\hat{R}_m &= \underset{R_m \in \mathcal{M}_Q([-1,1])}{\arg \min} \ \dfrac{1}{n} \left|\left| \bm{Y} - W \bm{Y} R_m - \bm{\widetilde{S}^m(X)} \begin{pmatrix} \hat{\mu}_m(R_m) \\ \hat{\beta}_m(R_m) \end{pmatrix} \right|\right|^2 \\
&= \underset{R_m \in \mathcal{M}_Q([-1,1])}{\arg \min} \ \dfrac{1}{n} \left|\left| \bm{Y} - W \bm{Y} R_m - \bm{\widetilde{S}^m(X)} \left( \bm{\widetilde{S}^m(X)}^\top \bm{\widetilde{S}^m(X)} + n \Lambda \right)^{-1} \left( \bm{\widetilde{S}^m(X)}^\top (\bm{Y} - W \bm{Y} R_m) \right) \right|\right|^2
\end{aligned}$ \Comment*[r]{Estimation of $R_m$} \vspace{0.2cm}
$\begin{pmatrix} \hat{\mu}_m \\ \hat{\beta}_m \end{pmatrix} = \begin{pmatrix} \hat{\mu}_m(\hat{R}_m) \\ \hat{\beta}_m(\hat{R}_m) \end{pmatrix} = \left( \bm{\widetilde{S}^m(X)}^\top \bm{\widetilde{S}^m(X)} + n \Lambda \right)^{-1} \left( \bm{\widetilde{S}^m(X)}^\top (\bm{Y} - W \bm{Y} \hat{R}_m) \right)$ \Comment*[r]{Estimation of $\mu_m$ and $\beta_m$ using (\ref{eq:functionofRm})}
\end{algorithm}
\end{footnotesize}

\begin{remark}
In practice the true parameter 
\begin{center}
$ m^* = \min \left\{ m \in \mathbb{N}^* / \exists (\mu_m^*,\beta_m^*,R_m^*) \in \mathcal{M}_{1 \times Q}(\mathbb{R}) \times \mathcal{B}_{s_P(m) \times Q, \alpha} \times \mathcal{M}_Q([-1,1]), \right.$ \\ $\left. \mathbb{E}\left[\bm{\mathcal{Y}} - W\bm{\mathcal{Y}}R_m^*  - \bm{1}_n \mu_m^*| \bm{\mathcal{X}}(.) \right] = \bm{S^m(\mathcal{X})}\beta_m^*  \right\} $
\end{center}
is unknown. However, as explained by \cite{fermanian2022functional}, since the balls $\{\mathcal{B}_{s_P(m) \times Q, \alpha} \}_{m \in \mathbb{N}^*}$ are nested, the function $L$ defined on $\mathbb{N}$ is decreasing on $\{1,\dots,m^*\}$ and is constant thereafter (equal to $\text{Tr}(\Sigma)$). \\
Its empirical version, 
$$ \widehat{L}(m) = \underset{\substack{\mu_m \in \mathcal{M}_{1 \times Q}(\mathbb{R}) \\ \beta_m \in \mathcal{B}_{s_P(m) \times Q, \alpha} \\ R_m \in \mathcal{M}_Q([-1,1])}}{\min} \ \widehat{\mathcal{R}}_m(\mu_m, \beta_m, R_m) $$ is however decreasing on $\mathbb{N}^*$ and to define an estimator $\widehat{m}$ of $m^*$, we must find a trade-off between a small value for the objective function and a relatively moderate number of coefficients $s_P(\widehat{m}) \times Q$ in $\widehat{\beta}_{\widehat{m}}$, i.e. a compromise between the objective and the complexity of the model. \\
\cite{fermanian2022functional} thus proposed to estimate $m^*$ by minimizing $\widehat{L}(m) + \mathrm{pen}_n(m)$ where $\mathrm{pen}_n$ penalizes the number of coefficients and is defined in Theorem \ref{theorem1}. If the minimum is reached in several values of $m$, the smallest is considered:
$$\widehat{m} = \min \ \left(\underset{m \in \mathbb{N}^*}{\arg \min} \ \widehat{L}(m) + \mathrm{pen}_n(m) \right). $$
This approach requires the parameters $K_{\mathrm{pen}}$ and $\kappa$ (see Theorem \ref{theorem1}) to be fixed. For $K_{\mathrm{pen}}$, we plot $\widehat{m} = \min (\underset{m \in \mathbb{N}^*}{\arg \min} (\widehat{L}(m) + \mathrm{pen}_n(m)))$ as a function of $K_{\mathrm{pen}}$ and get the value of $K_{\mathrm{pen}}$ that corresponds to the first big jump of $\widehat{m}$. Then $K_{\mathrm{pen}}$ is fixed to be twice this value \citep{birge2007minimal, fermanian2022functional}. For $\kappa$, \cite{fermanian2022functional} proposed to take $\kappa = 0.4$. \\
\end{remark}

\subsection{Theoretical guarantees}

In the following we assume $\mu_m = 0$ (and so we remove $\mu_m$ from the unknown parameters), which can be satisfied by centering $\bm{\mathcal{Y}}$ and $\bm{S^m(\mathcal{X})}$, and in addition to assumption ($\mathcal{H}_\alpha$), we assume 
$(\mathcal{H}_K)$:
\begin{itemize}
    \item[i.] $\exists K_{\mathcal{Y}}>0$ such that for all $i \in \{ 1, \dots, n\}, \ ||\mathcal{Y}_i|| \le K_{\mathcal{Y}}$ 
    \item[ii.] $\exists K_{\mathcal{X}}>0$ such that for all $i \in \{ 1, \dots, n\}, \ ||\mathcal{X}_i||_{TV} \le  K_{\mathcal{X}}$
    \item[iii.] $\exists K_\text{neighb} > 0$ such that for all $i \in \{ 1, \dots, n\}, \dsum_{j=1}^n \mathds{1}_{W_{i,j} \neq 0} \le K_\text{neighb}$ (each spatial unit have at most $K_\text{neighb}$ neighbors)
    \item[iv.] For all $i,j \in \{ 1, \dots, n\}, 0 \le W_{i,j} \le 1$
    \item[v.] For all $q,q' \in \{ 1, \dots, Q \}, |R_{m_{q,q'}}| \le 1$
\end{itemize}

It should be noted that these assumptions entail the following:
\begin{align*}
||\bm{\mathcal{Y}}|| &\le \sqrt{n} K_{\mathcal{Y}} \text{ by (i)} \\
||W_{i,\mydot}|| &\le \sqrt{K_\text{neighb}} \text{ ($W$ is bounded in rows) by (iii) and (iv)} \\
||W|| &\le \sqrt{n} \sqrt{K_\text{neighb}} \text{ by (iii) and (iv)} \\
||W_{i,\mydot} \bm{\mathcal{Y}} || &= \sqrt{ \dsum_{j=1}^Q \left( \dsum_{k=1}^n W_{i,k} \mathcal{Y}_{k,j}  \right)^2  } \text{ where } \mathcal{Y}_{k,j} \text{ denotes the } j^{th} \text{ variable of } \mathcal{Y}_k \\
&\le \sqrt{ \dsum_{j=1}^Q \left( \dsum_{k=1}^n W_{i,k} |\mathcal{Y}_{k,j}|  \right)^2  } \\
&\le \sqrt{ \dsum_{j=1}^Q \left( \dsum_{k=1}^n W_{i,k} K_{\mathcal{Y}}  \right)^2 } \text{ by (i)} \\
&= \sqrt{ \dsum_{j=1}^Q \left( \dsum_{k=1}^n W_{i,k} K_{\mathcal{Y}} \mathds{1}_{W_{i,k} \neq 0} \right)^2 } \\
&\le \sqrt{ \dsum_{j=1}^Q \left( K_\text{neighb} K_{\mathcal{Y}}  \right)^2 } = \sqrt{Q} K_\text{neighb} K_{\mathcal{Y}} \text{ by (iii) and (iv)} \\
||S^m(\mathcal{X}_i)|| &\le \left|\left|\widetilde{S}^m(\mathcal{X}_i)\right|\right| \le \exp{(||\mathcal{X}_i||_{TV})} \le \exp{(K_{\mathcal{X}})} \text{ by (ii) and Proposition 3 of \cite{fermanian2022functional}} \\
||R_m|| &\le Q \text{ by (v)}
\end{align*}

\begin{theorem} \label{theorem1}

Let $0<\kappa < \dfrac{1}{2}$, $\mathrm{pen}_n(m) = K_{\mathrm{pen}} n^{-\kappa} \sqrt{s_P(m)}$ and $n \geq \max (n_1,n_3)$ (where $n_1$ and $n_3$ are given in Propositions \ref{prop2} and \ref{prop4}), then

$$
\mathbb{P}(\widehat{m} \neq m^*) \le 148 m^* \exp{\left\{ -n\dfrac{K_4}{16} \left[ L(m^*-1)-\sigma^2 \right]^2\right\}} + 74\sum_{m > m^*}\exp{\left\{ - K_3 s_P(m) n^{-2\kappa+1} \right\}}.
$$

\end{theorem}

The proof is presented in Appendix \ref{appendix:proof}.

\section{Simulation study} \label{sec:simu}

A simulation study was conducted to evaluate the performances of the MPenSSAR and to compare them with the FSARLM \citep{ahmed2022quasi}, the ProjSSAR \citep{frevent2023functional} and the PenSSAR \citep{frevent2023functional}.

\subsection{Design of the simulation study}

We considered the case $Q=4$ and a grid with $60 \times 60$ locations, where we randomly allocate $n=200$ spatial units.

The outcome was generated by
$$ \bm{Y} = W \bm{Y} R + \bm{\theta} + \bm{e} $$ where $\bm{e} = (e_1^\top, \dots, e_n^\top)^\top$, $e_i \sim \mathcal{N}_4(\bm{0}_4, \Sigma)$, $\Sigma = \begin{pmatrix}
0.4 & 0.1 & 0.1 & 0.1 \\
0.1 & 0.4 & 0.1 & 0.1 \\
0.1 & 0.1 & 0.4 & 0.1 \\
0.1 & 0.1 & 0.1 & 0.4
\end{pmatrix}$ and $\bm{\theta} = (\theta_{i,q})_{\substack{1 \le i \le n \\ 1 \le q \le 4}}$. \\

Three simulation models were considered:
\begin{enumerate}
\item[(i)] $\theta_{i,q} = \dsum_{k=1}^{s_P(2)} S^2(X_i)_k \dfrac{\eta_{k,q}}{\sum_{k'=1}^{s_P(2)} \eta_{k',q}}, \eta_{k,q} \sim \mathcal{U}([0,1])$ (the true model),
\item[(ii)] $\theta_{i,q} = \dsum_{p=1}^P X_{i,p}(t_{101}) \dfrac{\eta_{p,q}}{\sum_{p'=1}^P \eta_{p',q}}, \eta_{p,q} \sim \mathcal{U}([0,1])$,
\item[(iii)] $\theta_{i,q} = 
\begin{cases} 
\dsum_{p=1}^P X_{i,p}(t_{101}) \dfrac{\eta_{p,q}}{\sum_{p'=1}^P \eta_{p',q}} &\text{if } q=1,3 \\
\dsum_{p=1}^P Z_{i,p}(t_{101}) \dfrac{\eta_{p,q}}{\sum_{p'=1}^P \eta_{p',q}} &\text{if } q=2,4 \end{cases}$ with $\eta_{p,q} \sim \mathcal{U}([0,1])$.
\end{enumerate}

The $X_i$ and $Z_i$ were generated as follows in 101 equally spaced times of $[0,1]$ ($t_1, \dots, t_{101}$) :
$$ X_i(t) = (X_{i,1}(t), \dots, X_{i,P}(t)), X_{i,p}(t) = \gamma_{i,p} t + f_{i,p}(t), \gamma_{i,p} \sim \mathcal{U}([-3,3]), $$
$$ Z_i(t) = (Z_{i,1}(t), \dots, Z_{i,P}(t)), Z_{i,p}(t) = \phi_{i,p} t + g_{i,p}(t), \phi_{i,p} \sim \mathcal{U}([-3,3]), $$
where $f_{i,p}$ and $g_{i,p}$ are two Gaussian processes with exponential covariance matrix with length-scale 1. \\

We considered $P=2$ and $P=10$, a spatial weight matrix $W$ constructed using the $8$-nearest neighbors method, and the following matrices $R$:
\begin{itemize}
\item[] $R_w = \begin{pmatrix*}[S]
 0.40 & -0.10 & 0.20 &  0.05 \\
-0.20 &  0.35 & 0.10 & -0.10 \\
 0.15 &  0.10 & 0.30 &  0.20 \\
 0.05 & -0.15 & 0.15 &  0.25
\end{pmatrix*}$ (weak spatial effects), 
\item[] $R_\text{mod} = \begin{pmatrix*}[S]
 0.6 & -0.2 & 0.4 &  0.2 \\
-0.4 &  0.6 & 0.2 & -0.2 \\
 0.3 &  0.2 & 0.5 &  0.4 \\
 0.1 & -0.3 & 0.3 &  0.4
\end{pmatrix*}$ (moderate spatial effects),
and
\item[] $R_h = \begin{pmatrix*}[S]
 0.9 & -0.6 & 0.7 & -0.7 \\
-0.8 &  0.7 & 0.8 &  0.6 \\
 0.6 &  0.7 & 0.7 &  0.9 \\
-0.7 &  0.8 & 0.7 &  0.6
\end{pmatrix*}$ (high spatial effects).
\end{itemize}

Then, we aim at predicting $Y_i$ given the observations of $X_i$ at times $t_1$ to $t_{101}$ for simulation (i), and the observations of $X_i$ at times $t_1$ to $t_{100}$ for simulations (ii) and (iii).

For each simulation study, and each value of $P$ and $R$, 100 datasets were generated and four approaches were compared:

\begin{itemize}
\item[(i)] The FSARLM proposed by \cite{ahmed2022quasi} on each $Y_{i,q}$ separately $(q=1,2,3,4)$, using a cubic B-splines basis with 12 equally spaced knots to approximate the $X_i$ from the observed data and a functional PCA \citep{ramsay2005functional}. As proposed by \cite{ahmed2022quasi}, we used a threshold on the number of coefficients such that the cumulative inertia was below 95\%.
\item[(ii)] The PenSSAR proposed by \cite{frevent2023functional} on each $Y_{i,q}$ separately $(q=1,2,3,4)$. 
\item[(iii)] The ProjSSAR proposed by \cite{frevent2023functional} on each $Y_{i,q}$ separately $(q=1,2,3,4)$, where a PCA was performed on the standardized truncated shifted-signature coefficients vectors, and similarly to \cite{ahmed2022quasi}, a threshold on the maximal number of coefficients such that the cumulative inertia was below 95\% was used.
\item[(iv)] Our new MPenSSAR approach.
\end{itemize}

It should be noted that signatures are invariant by translation and by time reparametrization \citep{lyons2007differential}. Thus, before computing the signature of $X_i$, we added an observation point taking the value 0 at the beginning of $X_i$ (this avoids the invariance by translation) and we considered $\widetilde{X}_i(t) = (X_i(t), t)$ (this avoids the invariance by time reparametrization). \\

Moreover, for the signature approaches (PenSSAR, ProjSSAR and MPenSSAR), the optimal truncation order $\widehat{m}$ was chosen on a validation set from a set $\{ 1, \dots, m_\text{max} \}$ of possible values where $m_\text{max}$ is such that $s_P(m_\text{max})$ is at most equal to $10^4$.
More generally, we split each dataset into a training, a validation and a test set, using an ordinary validation (OV) or a spatial validation (SV). For the latter we used a $K$-means algorithm (with $K = 6$) on the coordinates of the data and we randomly selected two clusters to be the validation and test sets. \\
Then, the optimal parameters (the number of coefficients for the FSARLM, $\widehat{m}$ for the PenSSAR and the MPenSSAR, and the optimal number of coefficients associated with $\widehat{m}$ for the ProjSSAR) were selected on the validation set using the root mean squared error (RMSE) criterion for the FSARLM, the ProjSSAR and the PenSSAR, and using $\widehat{L}$ for the MPenSSAR. Finally, the performances were measured by assessing the estimation of the matrix $R$, and the predictive capacity on the test set using the RMSE.

\subsection{Results of the simulation study}

The results are presented in Figure \ref{fig:estimY} and in Figures \ref{fig:estimRw}, \ref{fig:estimRm} and \ref{fig:estimRh} in Appendix \ref{appendix:ressim}. 

Considering the estimation of $R$, the FSARLM, ProjSSAR and PenSSAR approaches only estimate the diagonal coefficients since they consider the four variables in $Y$ separately. However, these coefficients are not always well estimated (particularly for $R_\text{mod}$ and $R_h$, see Figures \ref{fig:estimRm} and \ref{fig:estimRh} in Appendix \ref{appendix:ressim}), due to cross-variable spatial effects that are not taken into account. The MPenSSAR presents the advantage of estimating the whole $R$ matrix, which allows a better understanding of the spatial effects of each variable in $Y$ on itself, as well as the cross-variable spatial effects.

When considering the predictive capacity on the models, Figure \ref{fig:estimY} shows that in the case on weak spatial effects $(R_w)$, the MPenSSAR presents similar or lower performances than the other approaches. However when the spatial effects increase, it presents better RMSEs, especially when using a spatial validation.

\begin{figure}
\centering
\includegraphics[width=\textwidth]{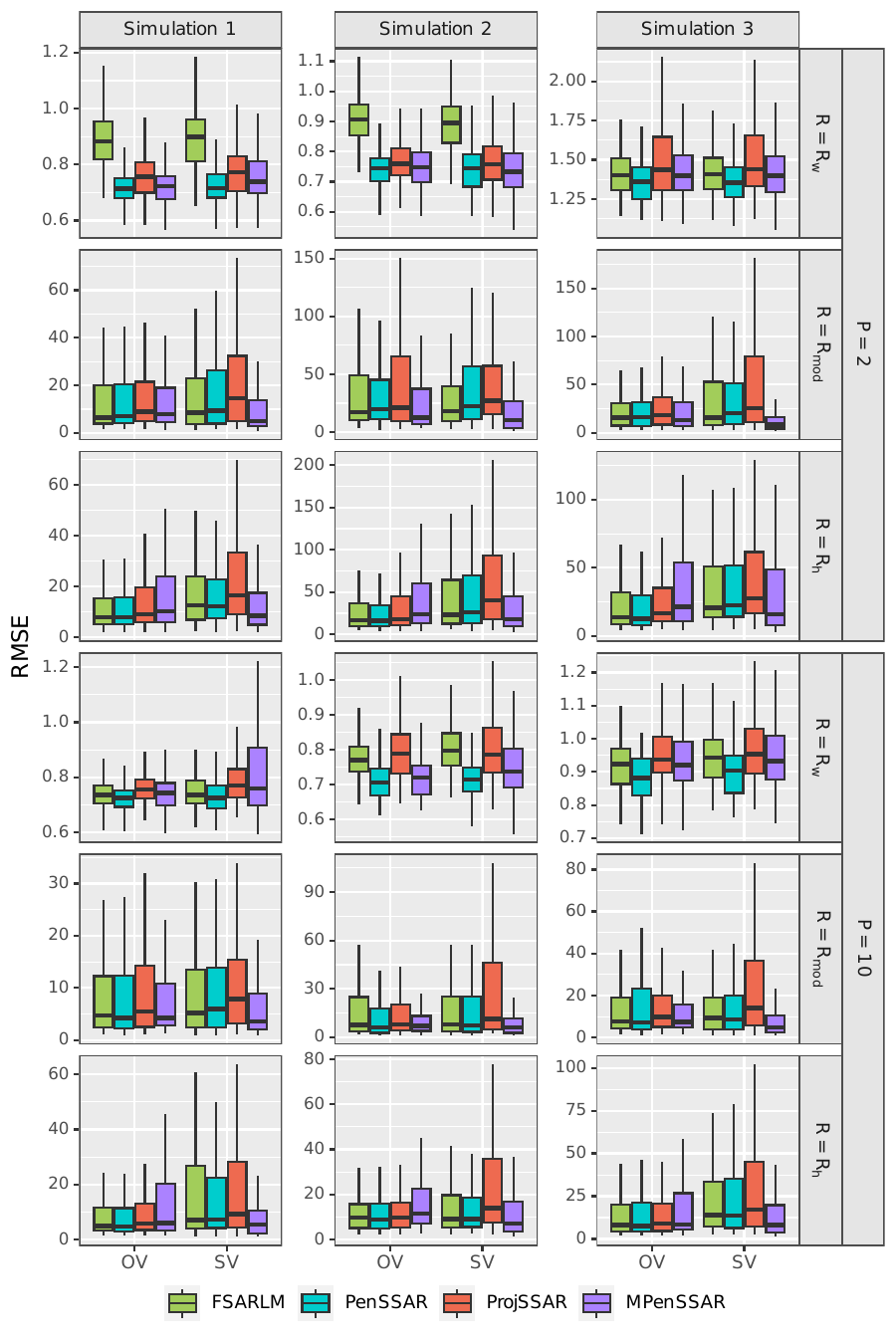}
\caption{RMSE on the test set with the FSARLM, the PenSSAR, the ProjSSAR and the MPenSSAR using ordinary (OV) and spatial (SV) validation}
\label{fig:estimY}
\end{figure}

\section{Real data application} \label{sec:appli}

As \cite{frevent2023functional}, we consider air quality data collected from 104 monitoring stations across the United States (\url{https://www.epa.gov/outdoor-air-quality-data}). The data consist in hourly nitrogen dioxide and ozone concentrations (in ppb) from August 1, 2022, 0:00 to August 4, 2022, 23:00.
We used linear interpolation to estimate the missing values. The spatial locations of the monitoring stations, as well as the ozone and nitrogen dioxide concentrations are presented in Figure \ref{fig:dataappli}.

\begin{figure}
\centering
\begin{minipage}{0.49\linewidth}
\includegraphics[width=\linewidth]{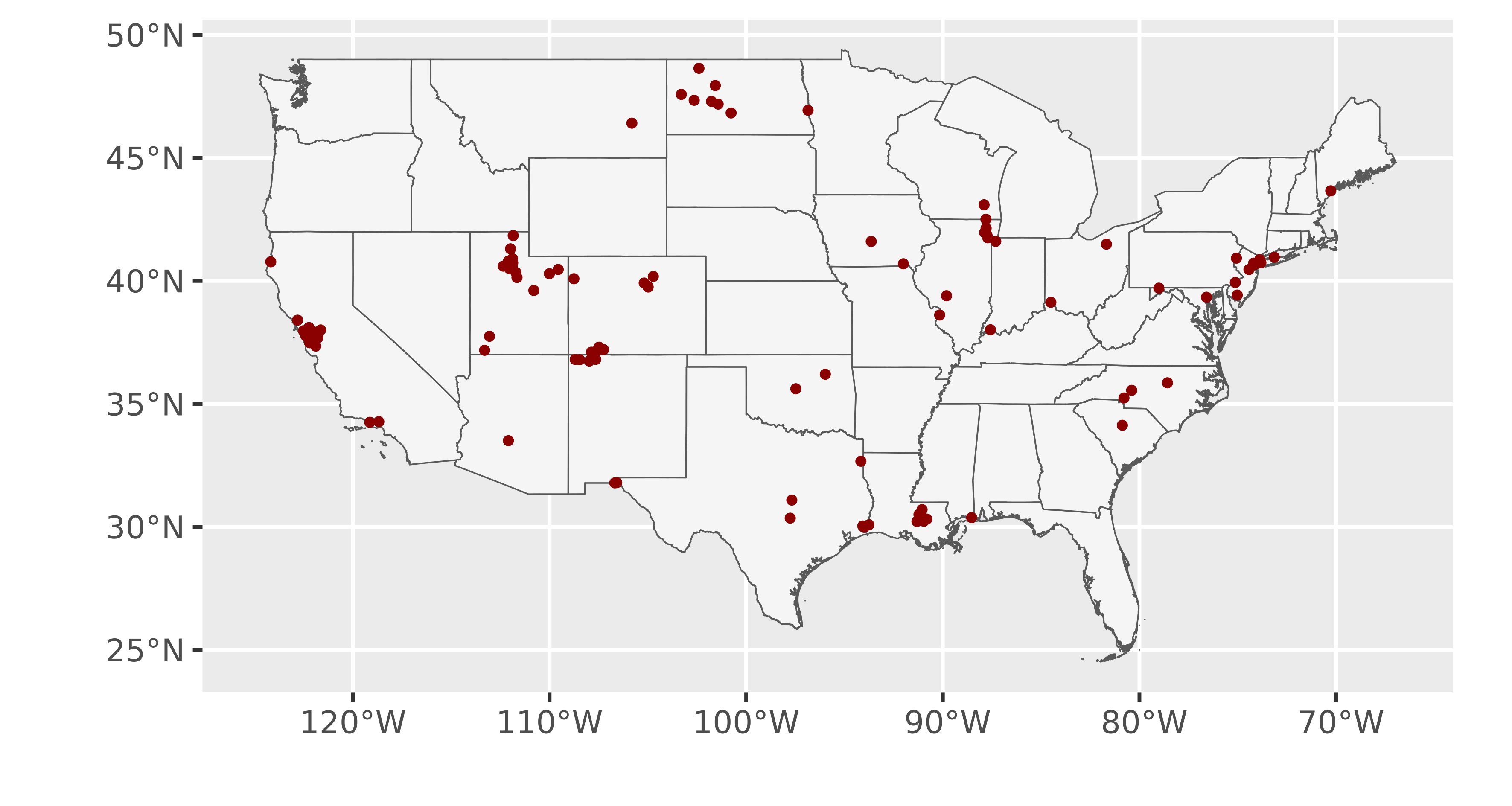}
\end{minipage}
\begin{minipage}{0.49\linewidth}
\includegraphics[width=\linewidth]{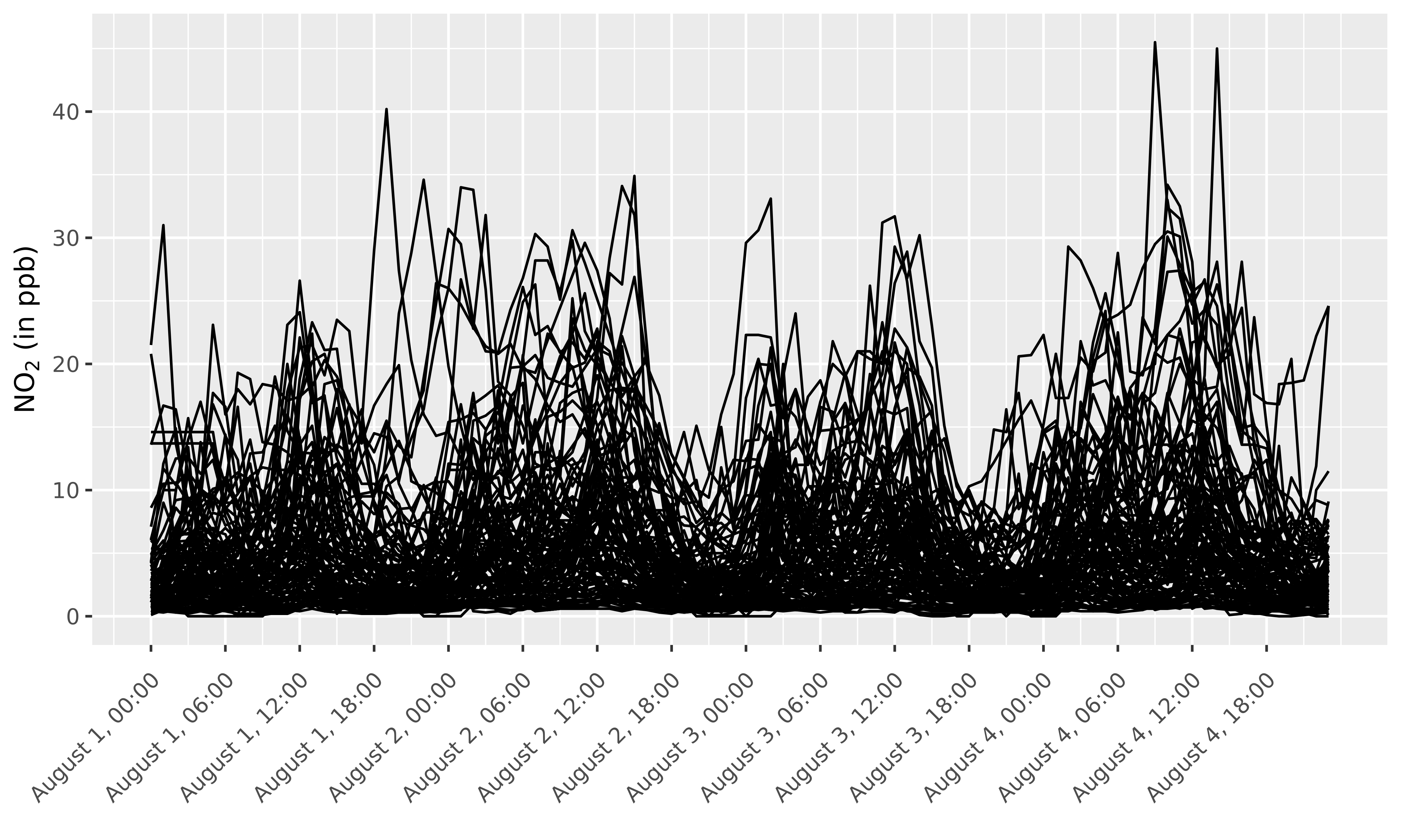}
\includegraphics[width=\linewidth]{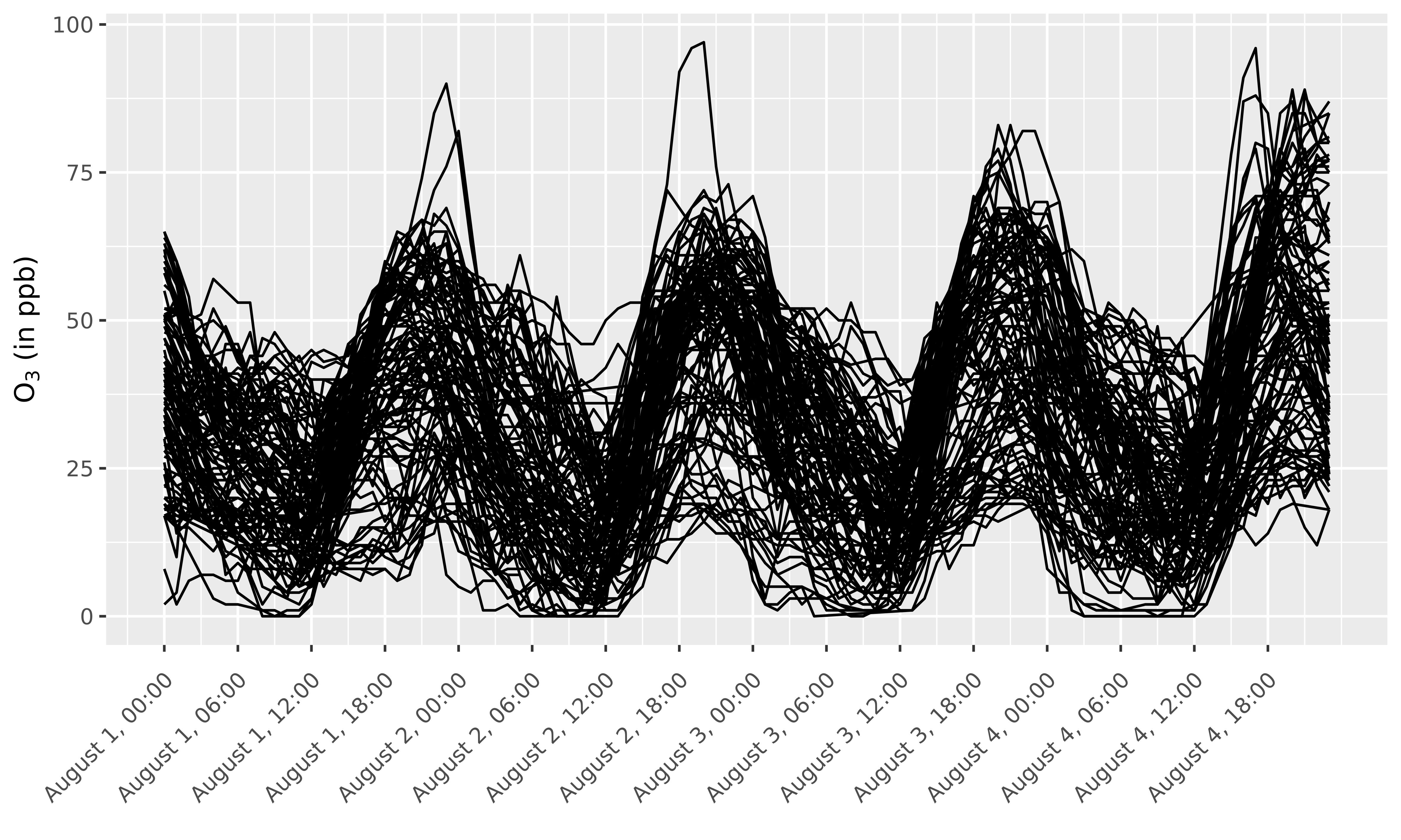}
\end{minipage}
\caption{Spatial locations of the 104 monitoring stations across the United States (left panel) and hourly nitrogen dioxide and ozone concentrations (from August 1, 2022, 0:00 to August 4, 2022, 23:00, right panel)
}
\label{fig:dataappli}
\end{figure}

We aim at predicting (i) the average concentration of nitrogen dioxide and ozone ($Q=2$) on August 4, 2022 from the nitrogen dioxide and ozone concentrations from August 1, 2022, 0:00 to August 3, 2022, 23:00, (ii) the maximum concentration of nitrogen dioxide and ozone on August 4, 2022 from the nitrogen dioxide and ozone concentrations from August 1, 2022, 0:00 to August 3, 2022, 23:00, (iii) the concentrations of nitrogen dioxide and ozone at 00:00 of August 4, 2022 from the nitrogen dioxide and ozone concentrations from August 1, 2022, 0:00 to August 3, 2022, 23:00, and (iv) the concentrations of nitrogen dioxide and ozone at 12:00 of August 4, 2022 from the nitrogen dioxide and ozone concentrations from August 1, 2022, 0:00 to August 4, 2022, 11:00.

We use the FSARLM, PenSSAR, ProjSSAR and MPenSSAR considering two spatial weight matrices: (i) spatial weights based on inverse distances $W_{i,j} = \left\{ \begin{array}{ll}
 \dfrac{1}{1+d_{ij}}    &  \text{if } d_{ij} < \tau \text{ and } i \neq j \\
  0   & \text{otherwise}
\end{array} \right.$ where $\tau$ is a threshold such that all monitoring stations have at least four neighbors \citep{ahmed2022quasi, frevent2023functional} and (ii) spatial weights based on the 4 nearest neighbors.

We consider ordinary validation and spatial validation using a $K$-means algorithm (with $K$ = 6) on the coordinates of the data, and we repeated the procedure on 30
different train/validation/test sets for each type of validation, thus covering all the possibilities for spatial validation.

Figure \ref{fig:resappli} presents the RMSE for the four objectives and Figure \ref{fig:Rappli} presents the estimation of $R$. 
The signature-based approaches (the PenSSAR, the ProjSSAR and the MPenSSAR) present similar or better RMSEs than the FSARLM. The MPenSSAR presents RMSEs comparable to the other signature-based approaches. However, it has the advantage of better estimating the spatial structure, as the MPenSSAR estimates the entire $R$ matrix and not just its diagonal elements, as other approaches do (see Figure \ref{fig:Rappli}).

\begin{figure}
\centering
\includegraphics[width=0.9\linewidth]{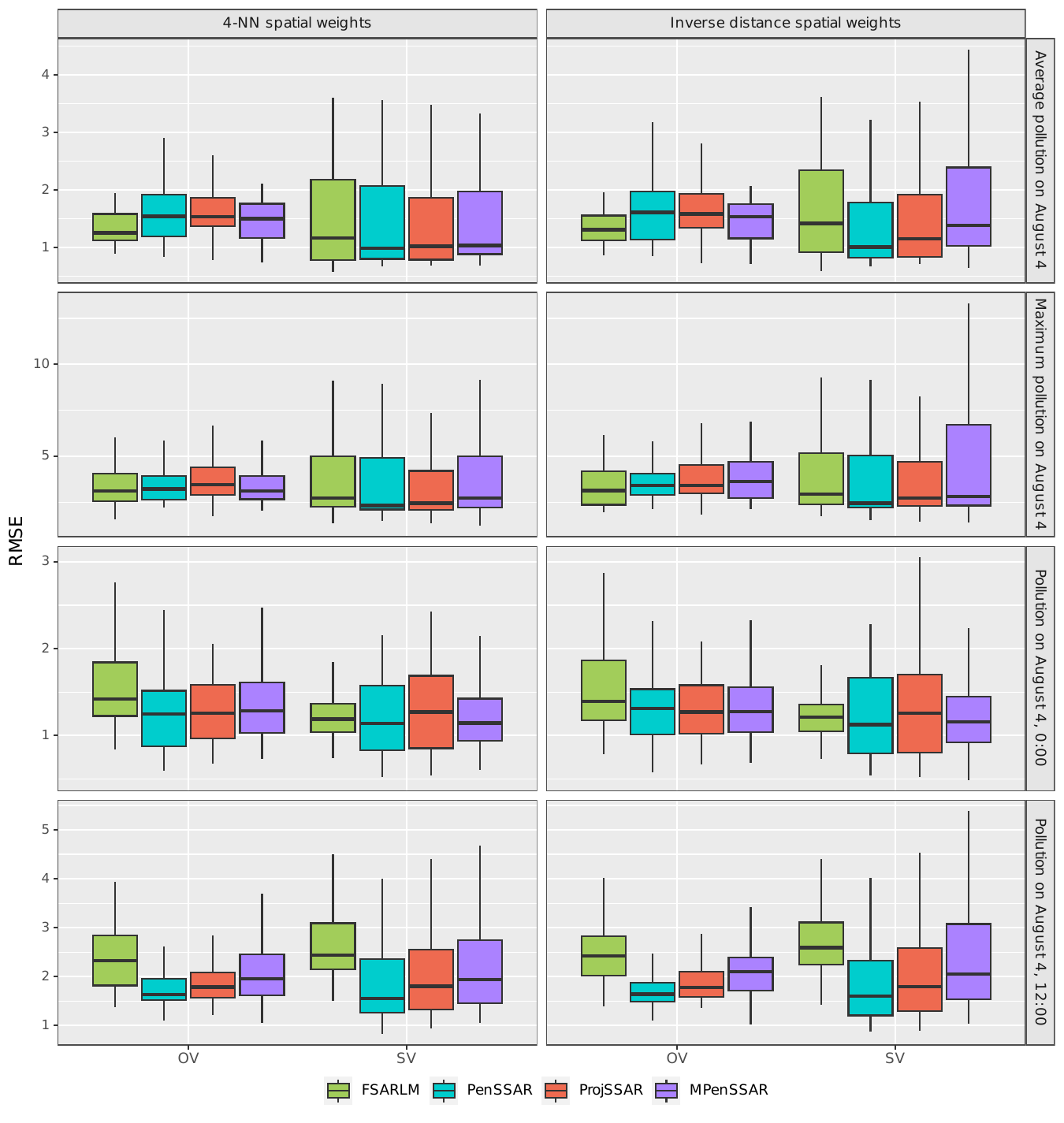}
\caption{RMSE on the test set with the FSARLM, the PenSSAR, the ProjSSAR and the MPenSSAR for predicting the concentrations of nitrogen dioxide and ozone using ordinary (OV) and spatial (SV) validation}
\label{fig:resappli}
\end{figure}

\begin{figure}
    \centering
\includegraphics[width=0.9\linewidth]{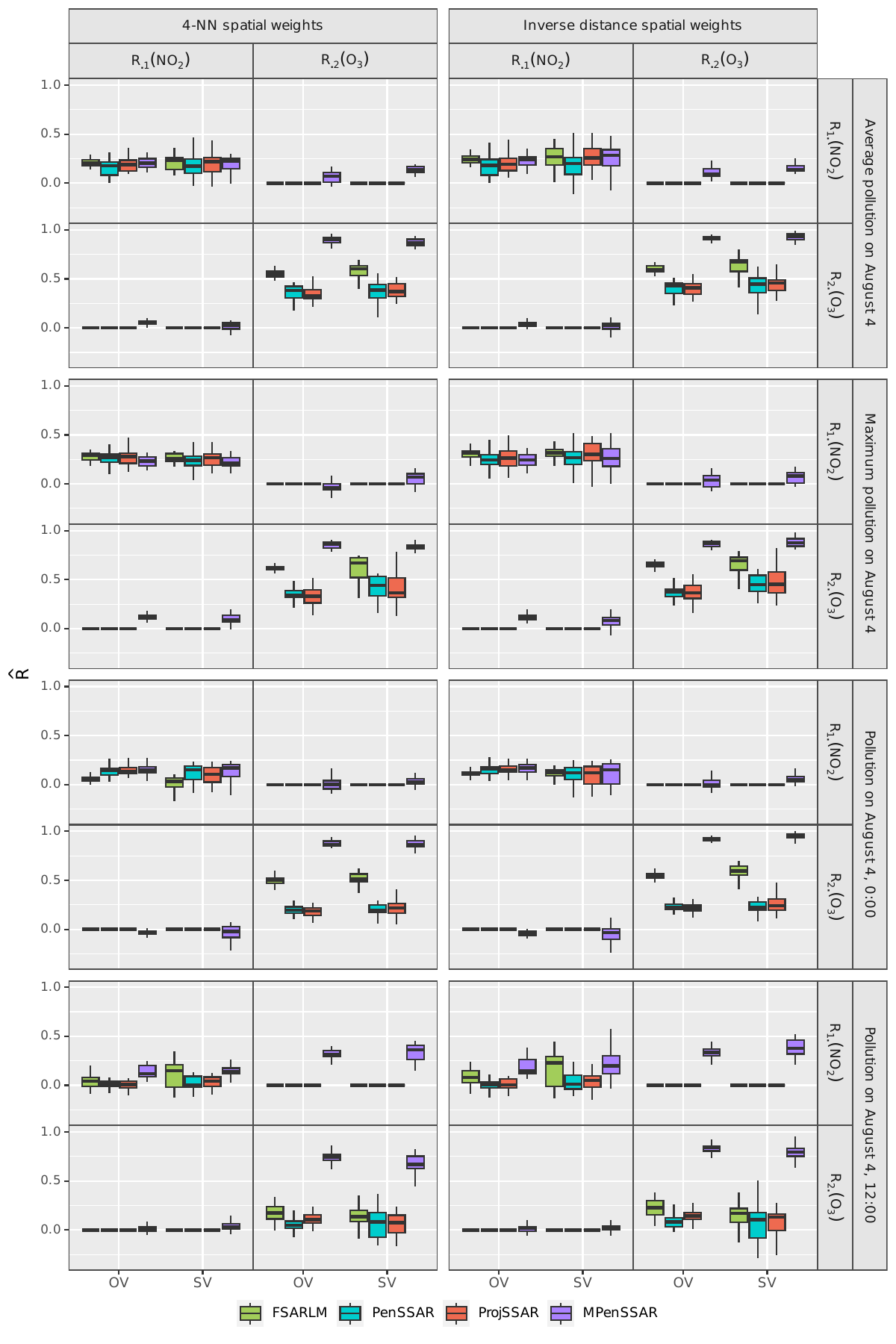}
\caption{Estimation of $R$ with the FSARLM, the PenSSAR, the ProjSSAR and the MPenSSAR for predicting the concentrations of nitrogen dioxide and ozone using ordinary (OV) and spatial (SV) validation}
\label{fig:Rappli}
\end{figure}

\section{Discussion} \label{sec:discussion}
This paper provides a study of the multivariate penalized signatures-based spatial regression (MPenSSAR) model. Our study builds upon a series of works on functional analysis based on signatures. The proposed model stands out from the existing literature by combining a multivariate response, the concept of signatures, and a spatial component. We have adapted the notion of signatures for the study of a multivariate response, which led us to provide a relevant estimator for the truncation order of the signature. \\ 

After presenting theoretical guarantees for our model, we conducted a simulation study where we showed that, contrary to the existing approaches in the literature, our new approach allows to accurately estimate the spatial effects of a variable on itself and the cross-variable spatial effects. It also performs well in prediction, especially when a spatial validation is used. \\

We then applied the MPenSSAR on a real data
set corresponding to ozone and nitrogen dioxide concentrations measured in monitoring stations across the United States. We showed that the MPenSSAR presents RMSEs comparable to other signature-based approaches but has the advantage of estimating all spatial effects, including, the cross-variable ones. \\

It should be noted that the proposed model proves to be quite flexible, and it is also suited to the non-spatial case. In fact, assuming that $R^* = 0$ and eliminating this term from the estimation procedure places us in the non-spatial framework, and the results proposed in this paper remain valid. \\

We hope this work will lead to further related research. Many other statistical models could be explored, and many other extensions could be considered. In particular, it would be interesting to study the effect of high dimensionality on our estimator and propose an estimator robust to dimensionality.

\bibliography{bibliographie.bib}
\bibliographystyle{chicago}

\appendix

\section{Proof of Theorem \ref{theorem1}} \label{appendix:proof}

\begin{lemma}[Hoeffding's lemma] \label{lemma:Hoeffdings} \phantom{linebreak} \\
Let $X$ a random variable such that $\mathbb{P}(a \le X \le b) = 1$, then $\forall \lambda \in \mathbb{R}$,
$$ \mathbb{E}\left\{\exp{[\lambda(X-\mathbb{E}(X))]}\right\} \le \exp{\left[ \dfrac{\lambda^2 (b-a)^2}{8} \right]}.$$
\vspace{0.3cm}
\end{lemma}

\begin{lemma}[Hoeffding's inequality] \label{lemma:Hoeffdingsineq} \phantom{linebreak} \\
Let $X_1, \dots, X_n$ be $n$ independent random variables such that $\forall i, \mathbb{P}(a_i \le X_i \le b_i) = 1$, then $\forall t > 0$,
$$ \mathbb{P}\left[\dsum_{i=1}^n \left( X_i - \mathbb{E}(X_i) \right) \ge t \right] \le \exp{\left[ \dfrac{- 2 t^2}{\dsum_{i=1}^n (b_i - a_i)^2 } \right]}.$$
\vspace{0.3cm}
\end{lemma}

\begin{definition}[Definition 5.5 from \cite{van2014probability}] \phantom{linebreak} \\
A set $N$ is a $\delta$-net for a metric space $(\mathcal{T},d)$ if for all $t \in \mathcal{T}$, there exists $\pi(t) \in N$ such that $d(t, \pi(t)) \le \delta$.
\end{definition}

\begin{theorem}[Theorem 5.29 from \cite{van2014probability}] \label{th:5.29} \phantom{linebreak} \\
Let $(X_t)_{t \in \mathcal{T}}$ a separable sub-Gaussian process on the metric space $(\mathcal{T}, d)$. Then $$\forall t' \in \mathcal{T}, x \ge 0, \mathbb{P}\left(\underset{t \in \mathcal{T}}{\sup} \ \{ X_t - X_{t'} \} \ge C \dint_{0}^{\infty} \sqrt{\log{N(\mathcal{T},d,\delta)}} \text{d}\delta + x \right) \le C \exp{\left(-\dfrac{x^2}{C \text{diam}(\mathcal{T})^2}\right)}$$
where $N(\mathcal{T},d,\delta) = \inf \ \{ |N| \ / N \text{ is a } \delta\text{-net for } (\mathcal{T},d) \}$.
\end{theorem}

\begin{lemma} \label{lemma:ineqN1} \phantom{linebreak} \\
Let $d((\beta_m, R_m),(\beta_m', R_m')) = d_1(\beta_m, \beta_m') + d_2(R_m, R_m')$ a distance on $\mathcal{B}_{s_P(m) \times Q, \alpha} \times \mathcal{M}_Q([-1,1])$. Then
$$ N\left(\mathcal{B}_{s_P(m) \times Q, \alpha} \times \mathcal{M}_Q([-1,1]), d, \delta \right) \le N\left(\mathcal{B}_{s_P(m) \times Q, \alpha}, d_1, \dfrac{\delta}{2}\right) N\left(\mathcal{M}_Q([-1,1]), d_2, \dfrac{\delta}{2}\right) $$ where $N(.)$ is defined in Theorem \ref{th:5.29}.
\end{lemma}
\begin{proof}
It suffices to show that if $N_1$ is a $\dfrac{\delta}{2}$-net for $(\mathcal{B}_{s_P(m) \times Q, \alpha}, d_1)$ and $N_2$ is a $\dfrac{\delta}{2}$-net for $(\mathcal{M}_Q([-1,1]), d_2)$, then $N_1 \times N_2$ is a $\delta$-net for $\left(\mathcal{B}_{s_P(m) \times Q, \alpha} \times \mathcal{M}_Q([-1,1]), d \right)$. \\

Let $N_1$ and $N_2$ be two $\dfrac{\delta}{2}$-nets for $(\mathcal{B}_{s_P(m) \times Q, \alpha}, d_1)$ and $(\mathcal{M}_Q([-1,1]), d_2)$ respectively. Let $(\beta_m, R_m) \in \mathcal{B}_{s_P(m) \times Q, \alpha} \times \mathcal{M}_Q([-1,1])$. \\
Since $N_1$ is a $\dfrac{\delta}{2}$-net for $(\mathcal{B}_{s_P(m) \times Q, \alpha}, d_1)$, there exists $\pi(\beta_m) \in N_1$ such that $d_1(\beta_m, \pi(\beta_m)) \le \dfrac{\delta}{2}$. \\
Since $N_2$ is a $\dfrac{\delta}{2}$-net for $(\mathcal{M}_Q([-1,1]), d_2)$, there exists $\pi(R_m) \in N_2$ such that $d_2(R_m, \pi(R_m)) \le \dfrac{\delta}{2}$. \\

Thus, there exists $\pi((\beta_m, R_m)) = (\pi(\beta_m), \pi(R_m)) \in N_1 \times N_2$ such that 
$$ d((\beta_m, R_m), \pi((\beta_m, R_m))) = d_1(\beta_m \pi(\beta_m)) + d_2(R_m, \pi(R_m)) \le \delta. $$
We deduce that $N_1 \times N_2$ is a $\delta$-net for $\left(\mathcal{B}_{s_P(m) \times Q, \alpha} \times \mathcal{M}_Q([-1,1]), d \right)$, which concludes the proof. \\
\end{proof}

\begin{lemma} \label{lemma:ineqN2} \phantom{linebreak} \\
Let $(\mathcal{A}_1, d)$ and $(\mathcal{A}_2, d)$ be two metric spaces such that $\mathcal{A}_1 \subset \mathcal{A}_2$. Then
$N(\mathcal{A}_1, d, \delta) \le N(\mathcal{A}_2, d, \delta)$, where $N(.)$ is defined in Theorem \ref{th:5.29}.
\end{lemma}
\begin{proof}
It suffices to show that if $N$ is a $\delta$-net for $(\mathcal{A}_2,d)$ then it is also a $\delta$-net for $(\mathcal{A}_1,d)$. \\

Let $N$ be a $\delta$-net for $(\mathcal{A}_2,d)$. Let $t \in \mathcal{A}_1$. \\
Since $\mathcal{A}_1 \subset \mathcal{A}_2$, $t \in \mathcal{A}_2$, and since $N$ be a $\delta$-net for $(\mathcal{A}_2,d)$, then there exists $\pi(t) \in N$ such that $d(t, \pi(t)) \le \delta$. \\
Thus $N$ is a $\delta$-net for $(\mathcal{A}_1,d)$ and this concludes the proof.
\end{proof}

\begin{lemma} [\cite{fermanian2022functional}] 
\label{lemma1fermanian} \phantom{linebreak} \\
For any $m \in \mathbb{N}$,
$$
\left|\widehat{L}(m)-L(m)\right| \le \sup_{\substack{\beta_m \in B_{s_P(m) \times Q, \alpha} \\ R_m \in \mathcal{M}_Q([-1,1])}}\left|\widehat{\mathcal{R}}_m(\beta_m, R_m)-\mathcal{R}_m(\beta_m, R_m)\right| .
$$
\end{lemma}

\begin{lemma}[\cite{fermanian2022functional}] \label{lemmafermanian} \phantom{linebreak} \\
For any $m > m^*$, $ \mathbb{P}(\widehat{m} = m) \le \mathbb{P}\left( 2 \underset{\substack{ \beta_m \in \mathcal{B}_{s_P(m) \times Q,\alpha} \\ R_m \in \mathcal{M}_{Q}([-1,1]) }}{\sup} \ |\widehat{\mathcal{R}}_{m}(\beta_m, R_m) - \mathcal{R}_m(\beta_m,R_m)| \ge \mathrm{pen}_n(m) - \mathrm{pen}_n(m^*) \right)$.
\end{lemma}

In the following, we consider
\begin{align*}
Z_m(\beta_m, R_m) &= \widehat{\mathcal{R}}_m(\beta_m,R_m) - \mathcal{R}_m(\beta_m,R_m)  \\ 
&= \dfrac{1}{n} \dsum_{i=1}^n \left[ \ || Y_i - (W\bm{Y}R_m)_{i,\mydot} - S^m(X_i)\beta_m ||^2 - \mathbb{E}\left( || \mathcal{Y}_i - (W\bm{\mathcal{Y}}R_m)_{i,\mydot} - S^m(\mathcal{X}_i)\beta_m ||^2 \right) \ \right].
\end{align*}

\begin{lemma} \label{lemma:Zmsub} \phantom{linebreak} \\
Under assumptions $(\mathcal{H}_\alpha)$ and $(\mathcal{H}_K)$, $\forall m \in \mathbb{N}, Z_m(\beta_m,R_m)_{\beta_m \in \mathcal{B}_{s_P(m) \times Q, \alpha}, R_m \in \mathcal{M}_Q([-1,1])}$ is sub-Gaussian for the distance 
$$ D'((\beta_m,R_m),(\beta_m',R_m')) = \dfrac{K}{\sqrt{n}} \Big[\sqrt{Q} K_\text{neighb} K_{\mathcal{Y}} ||R_m-R_m'|| + \exp{(K_{\mathcal{X}})} ||\beta_m - \beta_m'|| \ \Big], $$
where $K = 2 \left[K_{\mathcal{Y}} + K_\text{neighb} K_{\mathcal{Y}} Q^{\frac{3}{2}} + \exp{(K_{\mathcal{X}})} \alpha \right] $.
\end{lemma}

\vspace{0.3cm}

\begin{proof}
Since $\mathbb{E}[Z_m(\beta_m, R_m)]=0$, it suffices to show that
$$ \forall \lambda, \mathbb{E}\{\exp{[\lambda (Z_m(\beta_m,R_m) - Z_m(\beta_m', R_m')) ]} \} \le \exp{\left\{ \dfrac{\lambda^2 D'^2((\beta_m,R_m),(\beta_m',R_m'))}{2} \right\}} $$
for the metric $D'$. \\

Let $\ell_{\mathcal{X}_i, \mathcal{Y}_i}(\beta_m,R_m) = || \mathcal{Y}_i - (W\bm{\mathcal{Y}}R_m)_{i,\mydot} - S^m(\mathcal{X}_i)\beta_m ||^2$, then 
$$ Z_m(\beta_m, R_m) = \dfrac{1}{n} \dsum_{i=1}^n \big[ \ell_{X_i, Y_i}(\beta_m,R_m) - \mathbb{E}[\ell_{\mathcal{X}_i, \mathcal{Y}_i}(\beta_m,R_m)] \big]. $$

\textit{\textbf{Step 1: We show that $\ell_{\mathcal{X}, \mathcal{Y}}(\beta_m,R_m)$ is Lipschitz}} \\
We have to show that there exists $K \ge 0$ such that
$$ \left| \ell_{\mathcal{X}, \mathcal{Y}}(\beta_m,R_m) - \ell_{\mathcal{X}, \mathcal{Y}}(\beta_m',R_m') \right| \le K D((\beta_m,R_m),(\beta_m',R_m')) $$
for a metric $D$. \\

$$ |\ell_{\mathcal{X}_i, \mathcal{Y}_i}(\beta_m,R_m) - \ell_{\mathcal{X}_i, \mathcal{Y}_i}(\beta_m',R_m')| = | \ || \mathcal{Y}_i - (W\bm{\mathcal{Y}}R_m)_{i,\mydot} - S^m(\mathcal{X}_i)\beta_m ||^2 - || \mathcal{Y}_i - (W\bm{\mathcal{Y}}R_m')_{i,\mydot} - S^m(\mathcal{X}_i)\beta_m' ||^2 \ | . $$

Since $|a^2 - b^2| = |a+b| \ |a-b| \le 2 \max{(|a|,|b|)} \ |a-b|$, we have:
\begin{align*}
|\ell_{\mathcal{X}_i, \mathcal{Y}_i}(\beta_m,R_m) - \ell_{\mathcal{X}_i, \mathcal{Y}_i}(\beta_m',R_m')| 
&\le 2 \max{(|| \mathcal{Y}_i - (W\bm{\mathcal{Y}}R_m)_{i,\mydot} - S^m(\mathcal{X}_i)\beta_m ||,|| \mathcal{Y}_i - (W\bm{\mathcal{Y}}R_m')_{i,\mydot} - S^m(\mathcal{X}_i)\beta_m' ||)} \\ & \hspace{0.7cm} | \ || \mathcal{Y}_i - (W\bm{\mathcal{Y}}R_m)_{i,\mydot} - S^m(\mathcal{X}_i)\beta_m || - || \mathcal{Y}_i - (W\bm{\mathcal{Y}}R_m')_{i,\mydot} - S^m(\mathcal{X}_i)\beta_m' || \ | .
\end{align*}

\textbullet \ We consider $| \ || \mathcal{Y}_i - (W\bm{\mathcal{Y}}R_m)_{i,\mydot} - S^m(\mathcal{X}_i)\beta_m || - || \mathcal{Y}_i - (W\bm{\mathcal{Y}}R_m')_{i,\mydot} - S^m(\mathcal{X}_i)\beta_m' || \ |$: \\

Since $| \ ||a-b|| - ||a-c|| \ | \le ||b-c||$, we get (with $a = \mathcal{Y}_i$, $b = (W\bm{\mathcal{Y}}R_m)_{i,\mydot} + S^m(\mathcal{X}_i)\beta_m$ and $c = (W\bm{\mathcal{Y}}R_m')_{i,\mydot} + S^m(\mathcal{X}_i)\beta_m'$):
\begin{align*}
| \ || \mathcal{Y}_i - (W\bm{\mathcal{Y}}R_m)_{i,\mydot} - S^m(\mathcal{X}_i)\beta_m || &- || \mathcal{Y}_i - (W\bm{\mathcal{Y}}R_m')_{i,\mydot} - S^m(\mathcal{X}_i)\beta_m' || \ | \\
& \le || (W\bm{\mathcal{Y}}R_m)_{i,\mydot} + S^m(\mathcal{X}_i)\beta_m - (W\bm{\mathcal{Y}}R_m')_{i,\mydot} - S^m(\mathcal{X}_i)\beta_m'|| \\
&= || (W_{i,\mydot}\bm{\mathcal{Y}})(R_m-R_m') + S^m(\mathcal{X}_i)(\beta_m - \beta_m')|| \\
&\le ||W_{i,\mydot}\bm{\mathcal{Y}}|| \ ||R_m-R_m'|| + ||S^m(\mathcal{X}_i)|| \ || \beta_m - \beta_m' || .
\end{align*}

Now, since
$$ ||S^m(\mathcal{X}_i)|| \le \exp{( K_{\mathcal{X}} )} $$ 
and
$$ ||W_{i,\mydot} \bm{\mathcal{Y}}|| \le \sqrt{Q} K_\text{neighb} K_{\mathcal{Y}}, $$
we get
\begin{align*}
| \ || \mathcal{Y}_i - (W\bm{\mathcal{Y}}R_m)_{i,\mydot} - S^m(\mathcal{X}_i)\beta_m || &- || \mathcal{Y}_i - (W\bm{\mathcal{Y}}R_m')_{i,\mydot} - S^m(\mathcal{X}_i)\beta_m' || \ | \\
&\le ||W_{i,\mydot} \bm{\mathcal{Y}}|| \ ||R_m-R_m'|| + \exp{( K_{\mathcal{X}} )} \ || \beta_m - \beta_m' || \\ 
&\le  \sqrt{Q} K_\text{neighb} K_{\mathcal{Y}} \ ||R_m-R_m'|| + \exp{( K_{\mathcal{X}} )} \ || \beta_m - \beta_m' ||
\end{align*}

\vspace{0.3cm}

\textbullet \ We consider $\max{(|| \mathcal{Y}_i - (W\bm{\mathcal{Y}}R_m)_{i,\mydot} - S^m(\mathcal{X}_i)\beta_m ||,|| \mathcal{Y}_i - (W\bm{\mathcal{Y}}R_m')_{i,\mydot} - S^m(\mathcal{X}_i)\beta_m' ||)}$: \\
\begin{align*}
|| \mathcal{Y}_i - (W\bm{\mathcal{Y}}R_m)_{i,\mydot} - S^m(\mathcal{X}_i)\beta_m || &\le || \mathcal{Y}_i - (W\bm{\mathcal{Y}}R_m)_{i,\mydot}|| + || S^m(\mathcal{X}_i)\beta_m || \\
&\le || \mathcal{Y}_i|| + ||W_{i,\mydot} \bm{\mathcal{Y}}|| \ ||R_m|| + || S^m(\mathcal{X}_i)|| \ ||\beta_m || \\
&\le K_{\mathcal{Y}} + \sqrt{Q} K_\text{neighb} K_{\mathcal{Y}} Q + \exp{(K_{\mathcal{X}})} \alpha \\
&= K_{\mathcal{Y}} + K_\text{neighb} K_{\mathcal{Y}} Q^{\frac{3}{2}} + \exp{(K_{\mathcal{X}})} \alpha 
\end{align*}

Similarly, we can show
$$|| \mathcal{Y}_i - (W\bm{\mathcal{Y}}R_m')_{i,\mydot} - S^m(\mathcal{X}_i)\beta_m' || \le  K_{\mathcal{Y}} + K_\text{neighb} K_{\mathcal{Y}} Q^{\frac{3}{2}} + \exp{(K_{\mathcal{X}})} \alpha. $$

Thus,
\begin{align*} 
|\ell_{\mathcal{X}_i, \mathcal{Y}_i}(\beta_m,R_m) &- \ell_{\mathcal{X}_i, \mathcal{Y}_i}(\beta_m',R_m')| \\ &\le 2 \Big[K_{\mathcal{Y}} + K_\text{neighb} K_{\mathcal{Y}} Q^{\frac{3}{2}} + \exp{(K_{\mathcal{X}})} \alpha \Big] \ \Big[\sqrt{Q} K_\text{neighb} K_{\mathcal{Y}} ||R_m-R_m'|| +  \exp{(K_{\mathcal{X}})} ||\beta_m - \beta_m'|| \Big].
\end{align*}

Let $D((\beta_m,R_m),(\beta_m',R_m')) = \sqrt{Q} K_\text{neighb} K_{\mathcal{Y}} ||R_m-R_m'|| + \exp{(K_{\mathcal{X}})} ||\beta_m - \beta_m'||$ and $K = 2 [K_{\mathcal{Y}} + K_\text{neighb} K_{\mathcal{Y}} Q^{\frac{3}{2}} + \exp{(K_{\mathcal{X}})} \alpha] \ge 0$. $\ell_{\mathcal{X},\mathcal{Y}}$ is $K$-Lipschitz for the metric $D$. \\

\textit{\textbf{Step 2: Application of Hoeffding's lemma}} \\
We apply Lemma \ref{lemma:Hoeffdings} on $\mathcal{X}' = \ell_{\mathcal{X}, \mathcal{Y}}(\beta_m,R_m) - \ell_{\mathcal{X}, \mathcal{Y}}(\beta_m',R_m')$.

From Step 1, $|\mathcal{X}'| \le K D((\beta_m,R_m),(\beta_m',R_m'))$. \\
Thus, $\mathbb{P}(-K D((\beta_m,R_m),(\beta_m',R_m')) \le \mathcal{X}' \le K D((\beta_m,R_m),(\beta_m',R_m')))=1$. \\

We deduce 
$$ \forall \lambda \in \mathbb{R}, \mathbb{E}[\exp{(\lambda(\mathcal{X}'-\mathbb{E}(\mathcal{X}')))}] \le \exp{\left[ \dfrac{\lambda^2 (2KD((\beta_m,R_m),(\beta_m',R_m')))^2}{8} \right]} = \exp{\left[ \dfrac{\lambda^2 K^2 D^2((\beta_m,R_m),(\beta_m',R_m'))}{2} \right]} . $$

Now we denote $X'_i = \ell_{X_i, Y_i}(\beta_m,R_m) - \ell_{X_i, Y_i}(\beta_m',R_m')$. Noting that $\mathbb{E}[X'_i]=\mathbb{E}[\mathcal{X}'_i]$, we move to Step 3. \\

\textit{\textbf{Step 3: End of the proof}}
\begin{align*}
\mathbb{E}\{\exp{[\lambda (Z_m(\beta_m,R_m) - Z_m(\beta_m', R_m'))]} \} & \\
&\hspace{-3.5cm}= \mathbb{E}\left\{\exp{\left[\lambda \dfrac{1}{n} \dsum_{i=1}^n \big[\ell_{X_i,Y_i}(\beta_m,R_m) - \ell_{X_i,Y_i}(\beta_m',R_m')\big] - \mathbb{E}\big[ \ell_{\mathcal{X}_i,\mathcal{Y}_i}(\beta_m,R_m) - \ell_{\mathcal{X}_i,\mathcal{Y}_i}(\beta_m',R_m') \big]\right]} \right\} \\
&\hspace{-3.5cm}= \mathbb{E}\left\{\exp{\left[\lambda \dfrac{1}{n} \dsum_{i=1}^n (X'_i-\mathbb{E}(X'_i))\right]} \right\} \\
&\hspace{-3.5cm}= \dprod_{i=1}^n \mathbb{E}\left\{\exp{\left[\dfrac{\lambda}{n} (X'_i-\mathbb{E}(X'_i))\right]} \right\} \\
&\hspace{-3.5cm}\le \dprod_{i=1}^n \exp{\left[ \dfrac{\lambda^2K^2 D^2((\beta_m,R_m),(\beta_m',R_m'))}{2n^2} \right]} \\
&\hspace{-3.5cm}= \exp{\left[ \dfrac{\lambda^2 K^2 D^2((\beta_m,R_m),(\beta_m',R_m'))}{2n} \right]} .
\end{align*}

Let $D'((\beta_m,R_m),(\beta_m',R_m')) = \dfrac{K}{\sqrt{n}} D((\beta_m,R_m),(\beta_m',R_m'))$, we get
$$ \mathbb{E}\{\exp{[\lambda (Z_m(\beta_m,R_m) - Z_m(\beta_m', R_m'))]} \} 
\le \exp{\left[ \dfrac{\lambda^2 D'^2((\beta_m,R_m),(\beta_m',R_m'))}{2} \right]} $$
which completes the proof.

Then, $Z_m(\beta_m, R_m)$ is sub-Gaussian for the distance $D'$.

\end{proof}

\begin{proposition} \label{prop1} \phantom{linebreak} \\
Under assumptions $(\mathcal{H}_\alpha)$ and $(\mathcal{H}_K)$, 
$\forall m \in \mathbb{N}, x \ge 0, \beta_m' \in \mathcal{B}_{s_P(m) \times Q,\alpha}, R_m' \in \mathcal{M}_Q([-1,1])$
\begin{align*}
&\mathbb{P}\left(\underset{\substack{\beta_m \in \mathcal{B}_{s_P(m) \times Q,\alpha} \\ R_m \in \mathcal{M}_Q([-1,1])}}{\sup} \ Z_m(\beta_m, R_m) \ge 108 K \alpha \dfrac{1}{\sqrt{n}} \exp{(K_{\mathcal{X}})} \sqrt{s_P(m) \pi} + 108 K K_\text{neighb} K_{\mathcal{Y}} \dfrac{Q^{\frac{5}{2}}}{\sqrt{n}} \sqrt{\pi} + Z_m(\beta_m',R_m') + x\right) \\ 
&\hspace{3cm}\le 36 \exp{\left\{-\dfrac{x^2 n}{144K^2 [K_\text{neighb} K_{\mathcal{Y}} Q^{\frac{3}{2}} +  \exp{(K_{\mathcal{X}})} \alpha]^2}\right\}}.
\end{align*}
\end{proposition}

\vspace{0.3cm}

\begin{proof}
Let $m \in \mathbb{N}, x \ge 0, \beta_m' \in \mathcal{B}_{s_P(m) \times Q,\alpha}$ and $R_m' \in \mathcal{M}_Q([-1,1])$. \\

We apply Theorem \ref{th:5.29} on $Z_m$ which is sub-Gaussian for $D'$ from Lemma \ref{lemma:Zmsub}: \\

For the distance $D'$, $\text{diam}(\mathcal{B}_{s_P(m) \times Q,\alpha} \times \mathcal{M}_Q([-1,1])) = 2\dfrac{K}{\sqrt{n}} \Big( K_\text{neighb} K_{\mathcal{Y}} Q^{\frac{3}{2}} + \exp{(K_{\mathcal{X}})} \alpha \Big)$, then
\begin{align*}
&\mathbb{P}\left(\underset{\substack{\beta_m \in \mathcal{B}_{s_P(m) \times Q,\alpha} \\ R_m \in \mathcal{M}_Q([-1,1])} }{\sup} \ Z_m(\beta_m,R_m) - Z_m(\beta_m',R_m') \ge 36 \dint_0^\infty \sqrt{\log{N(\mathcal{B}_{s_P(m) \times Q,\alpha} \times \mathcal{M}_Q([-1,1]),D',\delta)}} \text{d}\delta + x\right) \\
&\hspace{3cm} \le 36 \exp{\left\{-\dfrac{x^2 n}{144K^2 [ K_\text{neighb} K_{\mathcal{Y}} Q^{\frac{3}{2}} + \exp{(K_{\mathcal{X}})} \alpha ]^2}\right\}} .   
\end{align*}

Now, from Lemma \ref{lemma:ineqN1}, 
$$ N\left(\mathcal{B}_{s_P(m) \times Q,\alpha} \times \mathcal{M}_Q([-1,1]),D',\delta\right) \le N\left(\mathcal{B}_{s_P(m) \times Q,\alpha} ,D'_1,\dfrac{\delta}{2}\right) N\left( \mathcal{M}_Q([-1,1]),D'_2,\dfrac{\delta}{2}\right), $$ 
and from Lemma \ref{lemma:ineqN2}, since $\mathcal{M}_Q([-1,1]) \subset \mathcal{B}_{Q \times Q,Q}$, 
$$ N\left( \mathcal{M}_Q([-1,1]),D'_2,\dfrac{\delta}{2}\right) \le N\left( \mathcal{B}_{Q \times Q,Q},D'_2,\dfrac{\delta}{2}\right).$$
Thus,
$$ N\left(\mathcal{B}_{s_P(m) \times Q,\alpha} \times \mathcal{M}_Q([-1,1]),D',\delta\right) \le N\left(\mathcal{B}_{s_P(m) \times Q,\alpha} ,D'_1,\dfrac{\delta}{2}\right) N\left( \mathcal{B}_{Q \times Q,Q},D'_2,\dfrac{\delta}{2}\right) $$ 
where 
$$D'_1(\beta_m,\beta_m')=\dfrac{K}{\sqrt{n}} \exp{(K_{\mathcal{X}})} \ || \beta_m - \beta_m'||$$
and $$ D'_2(R_m,R_m') = \dfrac{K}{\sqrt{n}} \sqrt{Q} K_\text{neighb} K_{\mathcal{Y}} ||R_m-R_m'||.$$

Next, from \cite[Lemma 5.13]{van2014probability} we have 
\begin{alignat*}{3}
N\left(\mathcal{B}_{s_P(m) \times Q,\alpha} ,D'_1,\dfrac{\delta}{2}\right) &\le \left( \dfrac{6 K \alpha \exp{(K_{\mathcal{X}})}}{\sqrt{n} \delta} \right)^{s_P(m)} && \text{if } 0 < \dfrac{\sqrt{n} \delta \exp{(-K_{\mathcal{X}})}}{2K\alpha} < 1 \\
N\left(\mathcal{B}_{s_P(m) \times Q,\alpha} ,D'_1,\dfrac{\delta}{2} \right) &= 1 && \text{if } \sqrt{n}\delta \ge 2 K \alpha \exp{(K_{\mathcal{X}})}
\end{alignat*} 
and
\begin{alignat*}{3}
N\left( \mathcal{B}_{Q \times Q,Q},D'_2,\dfrac{\delta}{2}\right) &\le \left( \dfrac{6 Q^{\frac{3}{2}} K K_\text{neighb} K_{\mathcal{Y}}}{\sqrt{n} \delta} \right)^{Q^2} && \text{if } 0 < \dfrac{\sqrt{n} \delta}{2 Q^{\frac{3}{2}} K K_\text{neighb} K_{\mathcal{Y}}} < 1 \\
N\left( \mathcal{B}_{Q \times Q,Q},D'_2,\dfrac{\delta}{2}\right) &= 1 && \text{if } \sqrt{n} \delta \ge 2 Q^{\frac{3}{2}} K K_\text{neighb} K_{\mathcal{Y}}.
\end{alignat*}

\vspace{0.3cm}

\textit{\textbf{Situation 1:}} $Q^{\frac{3}{2}} K_\text{neighb} K_{\mathcal{Y}} \le \alpha  \exp{(K_{\mathcal{X}})}$
\begin{align*}
&\dint_0^\infty \sqrt{\log{N(\mathcal{B}_{s_P(m) \times Q,\alpha} \times \mathcal{M}_Q([-1,1]),D',\delta)}} \text{d}\delta \\
&\le \dint_0^{2 Q^{\frac{3}{2}} K K_\text{neighb} K_{\mathcal{Y}} / \sqrt{n}} \sqrt{s_P(m) \log{\left[ \dfrac{6K\alpha \exp{(K_{\mathcal{X}})}}{\sqrt{n}\delta} \right]} + Q^2 \log{\left[ \dfrac{6Q^{\frac{3}{2}} K K_\text{neighb} K_{\mathcal{Y}}}{\sqrt{n} \delta} \right]}
} \text{d}\delta \\
&+ \dint_{2 Q^{\frac{3}{2}} K K_\text{neighb} K_{\mathcal{Y}} / \sqrt{n}}^{2K\alpha \exp{(K_{\mathcal{X}})}/\sqrt{n}} \sqrt{
s_P(m) \log{\left( \dfrac{6K\alpha \exp{(K_{\mathcal{X}})}}{\sqrt{n}\delta} \right)}} \text{d}\delta \\
&\le \dint_{0}^{2K\alpha \exp{(K_{\mathcal{X}})}/\sqrt{n}} \sqrt{
s_P(m) \log{\left( \dfrac{6K\alpha  \exp{(K_{\mathcal{X}})}}{\sqrt{n}\delta} \right)}} \text{d}\delta + \dint_0^{2 Q^{\frac{3}{2}} K K_\text{neighb} K_{\mathcal{Y}} / \sqrt{n}} \sqrt{Q^2 \log{\left( \dfrac{6 Q^{\frac{3}{2}} K K_\text{neighb} K_{\mathcal{Y}}}{\sqrt{n} \delta} \right)}
} \text{d}\delta \\
&\le 3K\dfrac{\alpha}{\sqrt{n}} \exp{(K_{\mathcal{X}})} \sqrt{s_P(m) \pi} + 3 K K_\text{neighb} K_{\mathcal{Y}} \dfrac{Q^{\frac{5}{2}}}{\sqrt{n}} \sqrt{\pi} 
\end{align*}

\vspace{0.3cm}

\textit{\textbf{Situation 2:}} $Q^{\frac{3}{2}} K_\text{neighb} K_{\mathcal{Y}} \ge \alpha  \exp{(K_{\mathcal{X}})}$ \\
We have the same inequality. \\

Thus 
$$ 36 \dint_0^\infty \sqrt{\log{N(\mathcal{B}_{s_P(m) \times Q,\alpha} \times \mathcal{M}_Q([-1,1]),D',\delta)}} \text{d}\delta \le 108 K \alpha \dfrac{1}{\sqrt{n}} \exp{(K_{\mathcal{X}})} \sqrt{s_P(m) \pi} + 108 K K_\text{neighb} K_{\mathcal{Y}} \dfrac{Q^{\frac{5}{2}}}{\sqrt{n}} \sqrt{\pi}.$$

Finally,
\begin{align*}
&\mathbb{P}\left(\underset{\substack{\beta_m \in \mathcal{B}_{s_P(m) \times Q,\alpha} \\ R_m \in \mathcal{M}_Q([-1,1])}}{\sup} \ Z_m(\beta_m,R_m) \ge 108 K \alpha \dfrac{1}{\sqrt{n}} \exp{(K_{\mathcal{X}})} \sqrt{s_P(m) \pi} + 108  K K_\text{neighb} K_{\mathcal{Y}} \dfrac{Q^{\frac{5}{2}}}{\sqrt{n}} \sqrt{\pi}
+ Z_m(\beta_m', R_m') + x\right) \\
&\hspace{3cm} \le 36 \exp{\left\{-\dfrac{x^2 n}{144K^2 [ K_\text{neighb} K_{\mathcal{Y}} Q^{\frac{3}{2}} +  \exp{(K_{\mathcal{X}})} \alpha ]^2}\right\}}.    
\end{align*}
\end{proof}

\begin{proposition} \label{prop2} \phantom{linebreak} \\
Let $0<\kappa < \dfrac{1}{2}$ and $\mathrm{pen}_n(m) = K_{\mathrm{pen}} n^{-\kappa} \sqrt{s_P(m)}$, $n_1$ the smallest integer such that $$n_1 \ge \left\{ \dfrac{\sqrt{s_P(m^*+1)}-\sqrt{s_P(m^*)}}{\sqrt{s_P(m^*+1)}} \dfrac{K_{\mathrm{pen}}}{864 K\sqrt{\pi} \left[ \alpha \exp{(K_{\mathcal{X}})} + K_\text{neighb} K_{\mathcal{Y}} Q^{\frac{5}{2}} s_P(m^*+1)^{-\frac{1}{2}} \right]} \right\}^{\frac{1}{\kappa-\frac{1}{2}}}.$$

Then, under $(\mathcal{H}_\alpha)$ and $(\mathcal{H}_K)$, $\forall m > m^*, n \ge n_1$, 
$$ \mathbb{P}(\widehat{m}=m) \le 74 \exp{\left[ - K_3 s_P(m) n^{-2\kappa+1} \right]},$$ where
$$ K_3 = \left(1-\sqrt{\dfrac{s_P(m^*)}{s_P(m^*+1)}}\right)^2 \dfrac{K^2_{\mathrm{pen}}}{8} \min{\left\{ \dfrac{1}{K_{\mathcal{Y}}^4} , \dfrac{1}{1152K^2 \left[ K_\text{neighb} K_{\mathcal{Y}} Q^{\frac{3}{2}} + \exp{(K_{\mathcal{X}})}\alpha \right]^2} \right\}}. $$
\end{proposition}

\vspace{0.3cm}

\begin{proof}
Let $u_{m,n} = \dfrac{1}{2}\left[ \ \mathrm{pen}_n(m)-\mathrm{pen}_n(m^*) \ \right] = \dfrac{1}{2} K_{\mathrm{pen}} n^{-\kappa} \left[ \ \sqrt{s_P(m)}-\sqrt{s_P(m^*)} \ \right] $. \\

From Lemma \ref{lemmafermanian}, we have
\begin{align*}
\forall m > m^*, \mathbb{P}(\widehat{m}=m) &\le \mathbb{P}\left(2 \underset{\substack{\beta_m \in \mathcal{B}_{s_P(m) \times Q,\alpha} \\ R_m \in \mathcal{M}_Q([-1,1]) }}{\sup} \ | \widehat{\mathcal{R}}_m(\beta_m,R_m) - \mathcal{R}_m(\beta_m,R_m) | \ge 2 u_{m,n} \right) \\ &= 
\mathbb{P}\left(\underset{\substack{\beta_m \in \mathcal{B}_{s_P(m) \times Q,\alpha} \\ R_m \in \mathcal{M}_Q([-1,1]) }}{\sup} \ | Z_m(\beta_m,R_m) | \ge u_{m,n} \right).
\end{align*}

We also have
\begin{align*}
\mathbb{P}\left(\underset{\substack{\beta_m \in \mathcal{B}_{s_P(m) \times Q,\alpha} \\ R_m \in \mathcal{M}_Q([-1,1]) }}{\sup} \ | Z_m(\beta_m,R_m) | \ge u_{m,n} \right) &\le \mathbb{P}\left(\underset{\substack{\beta_m \in \mathcal{B}_{s_P(m) \times Q,\alpha} \\ R_m \in \mathcal{M}_Q([-1,1]) }}{\sup} \ Z_m(\beta_m,R_m) \ge u_{m,n} \right) \\ &+\mathbb{P}\left(\underset{\substack{\beta_m \in \mathcal{B}_{s_P(m) \times Q,\alpha} \\ R_m \in \mathcal{M}_Q([-1,1]) }}{\sup} \ -Z_m(\beta_m,R_m) \ge u_{m,n} \right),
\end{align*}

where 
\begin{align*}
\mathbb{P}\left(\underset{\substack{\beta_m \in \mathcal{B}_{s_P(m) \times Q,\alpha} \\ R_m \in \mathcal{M}_Q([-1,1]) }}{\sup} \ Z_m(\beta_m,R_m) \ge u_{m,n} \right) &= \mathbb{P}\left(\underset{\substack{\beta_m \in \mathcal{B}_{s_P(m),\alpha} \\ R_m \in \mathcal{M}_Q([-1,1]) }}{\sup} \ Z_m(\beta_m,R_m) \ge u_{m,n}, Z_m(\beta_m',R_m') \le \dfrac{u_{m,n}}{2} \right) \\ &+ \mathbb{P}\left(\underset{\substack{\beta_m \in \mathcal{B}_{s_P(m) \times Q,\alpha} \\ R_m \in \mathcal{M}_Q([-1,1]) }}{\sup} \ Z_m(\beta_m,R_m) \ge u_{m,n}, Z_m(\beta_m',R_m') > \dfrac{u_{m,n}}{2} \right) \\
&\le \underbrace{\mathbb{P}\left(\underset{\substack{\beta_m \in \mathcal{B}_{s_P(m) \times Q,\alpha} \\ R_m \in \mathcal{M}_Q([-1,1]) }}{\sup} \ Z_m(\beta_m,R_m) \ge Z_m(\beta_m',R_m') + \dfrac{u_{m,n}}{2} \right)}_{a} \\ &+ \underbrace{\mathbb{P}\left(Z_m(\beta_m',R_m') > \dfrac{u_{m,n}}{2} \right)}_{b} .
\end{align*}

\textbullet \ To apply Proposition \ref{prop1} to $a$ with $x=\dfrac{u_{m,n}}{2}-108 K \alpha \dfrac{1}{\sqrt{n}} \exp{(K_{\mathcal{X}})} \sqrt{s_P(m) \pi } - 108 K K_\text{neighb} K_{\mathcal{Y}} \dfrac{Q^{\frac{5}{2}}}{\sqrt{n}} \sqrt{\pi}$, we need $x \ge 0$. \\
\begin{align*}
x &= \dfrac{1}{4} K_{\mathrm{pen}} n^{-\kappa} \left( \sqrt{s_P(m)}-\sqrt{s_P(m^*)} \right) -108 K \alpha \dfrac{1}{\sqrt{n}} \exp{(K_{\mathcal{X}})} \sqrt{s_P(m) \pi } - 108 K K_\text{neighb} K_{\mathcal{Y}} \dfrac{Q^{\frac{5}{2}}}{\sqrt{n}} \sqrt{\pi} \\
&= \dfrac{1}{4} K_{\mathrm{pen}} n^{-\kappa} \sqrt{s_P(m)} \left[ 1-\sqrt{\dfrac{s_P(m^*)}{s_P(m)}}
 -432 \dfrac{K}{K_{\mathrm{pen}}} n^{\kappa-\frac{1}{2}} \alpha \sqrt{\pi} \exp{(K_{\mathcal{X}})} - 432 \dfrac{K K_\text{neighb} K_{\mathcal{Y}}}{K_{\mathrm{pen}}} n^{\kappa-\frac{1}{2}} Q^{\frac{5}{2}} \sqrt{\dfrac{\pi}{s_P(m)}}
\right] \\
&\ge \dfrac{K_{\mathrm{pen}}}{4} n^{-\kappa} \sqrt{s_P(m)} \left[ 1-\sqrt{\dfrac{s_P(m^*)}{s_P(m^*+1)}}
 -432 \dfrac{K \alpha}{K_{\mathrm{pen}}} n^{\kappa-\frac{1}{2}}  \sqrt{\pi} \exp{(K_{\mathcal{X}})} \right.
 \\ &\hspace{6.8cm} \left. - 432 \dfrac{K K_\text{neighb} K_{\mathcal{Y}}}{K_{\mathrm{pen}}} n^{\kappa-\frac{1}{2}} Q^{\frac{5}{2}} \sqrt{\dfrac{\pi}{s_P(m^*+1)}}
\right] \\
\end{align*}

Since $\kappa < \dfrac{1}{2}$,
the right term is increasing with $n$ and we must have
$n \ge n_1$ with $n_1$ such that it is positive:

\begin{align*}
&1-\sqrt{\dfrac{s_P(m^*)}{s_P(m^*+1)}}
 -432 \dfrac{K}{K_{\mathrm{pen}}} n^{\kappa-\frac{1}{2}} \alpha \sqrt{\pi} \exp{(K_{\mathcal{X}})} - 432 \dfrac{K K_\text{neighb} K_{\mathcal{Y}}}{K_{\mathrm{pen}}} n^{\kappa-\frac{1}{2}} Q^{\frac{5}{2}} \sqrt{\dfrac{\pi}{s_P(m^*+1)}} \ge 0 \\
 &\iff n^{\kappa-\frac{1}{2}} \left[
 432 \dfrac{K}{K_{\mathrm{pen}}} \alpha \sqrt{\pi} \exp{(K_{\mathcal{X}})} + 432 \dfrac{K K_\text{neighb} K_{\mathcal{Y}}}{K_{\mathrm{pen}}}  Q^{\frac{5}{2}} \sqrt{\dfrac{\pi}{s_P(m^*+1)}} \right] \le 1-\sqrt{\dfrac{s_P(m^*)}{s_P(m^*+1)}} \\
  &\iff n \ge \left\{ \dfrac{\sqrt{s_P(m^*+1)}-\sqrt{s_P(m^*)}}{\sqrt{s_P(m^*+1)}} \dfrac{K_{\mathrm{pen}}}{432K\sqrt{\pi} \left[ \alpha \exp{(K_{\mathcal{X}})} +  K_\text{neighb} K_{\mathcal{Y}} Q^{\frac{5}{2}} s_P(m^*+1)^{-\frac{1}{2}} \right]} \right\}^{\frac{1}{\kappa-\frac{1}{2}}} \\
\end{align*}

Now consider $n_1 = \left\lceil \left\{ \dfrac{\sqrt{s_P(m^*+1)}-\sqrt{s_P(m^*)}}{\sqrt{s_P(m^*+1)}} \dfrac{K_{\mathrm{pen}}}{864 K\sqrt{\pi} \left[ \alpha \exp{(K_{\mathcal{X}})} + K_\text{neighb} K_{\mathcal{Y}} Q^{\frac{5}{2}} s_P(m^*+1)^{-\frac{1}{2}} \right]} \right\}^{\frac{1}{\kappa-\frac{1}{2}}} \right\rceil$, \\
then for $n \ge n_1$,
\begin{align*}
x &= \dfrac{u_{m,n}}{2}-108 K \alpha \dfrac{1}{\sqrt{n}} \exp{(K_{\mathcal{X}})} \sqrt{s_P(m) \pi } - 108 K K_\text{neighb} K_{\mathcal{Y}} \dfrac{Q^{\frac{5}{2}}}{\sqrt{n}} \sqrt{\pi} \\
&\ge \dfrac{1}{4} K_{\mathrm{pen}} n^{-\kappa} \sqrt{s_P(m)} \dfrac{1}{2} \left( 1 - \sqrt{\dfrac{s_P(m^*)}{s_P(m^*+1)}}
 \right) = \dfrac{1}{8} K_{\mathrm{pen}} n^{-\kappa} \sqrt{s_P(m)} \left( 1 - \sqrt{\dfrac{s_P(m^*)}{s_P(m^*+1)}}
 \right) \\
 &\ge 0 . 
\end{align*}

Thus we can apply Proposition \ref{prop1}:

$$
\mathbb{P}\left(\underset{\substack{\beta_m \in \mathcal{B}_{s_P(m) \times Q,\alpha} \\ R_m \in \mathcal{M}_Q([-1,1])}}{\sup} \ Z_m(\beta_m,R_m) \ge \dfrac{u_{m,n}}{2}
+ Z_m(\beta_m',R_m') \right)
\le 36 \exp{\left\{ -\dfrac{x^2 n}{144K^2 \left[ K_\text{neighb} K_{\mathcal{Y}} Q^{\frac{3}{2}} +  \exp{(K_{\mathcal{X}})} \alpha \right]^2}\right\}},    $$

where
$$
x \ge \dfrac{1}{8} K_{\mathrm{pen}} n^{-\kappa} \sqrt{s_P(m)} \left( 1 - \sqrt{\dfrac{s_P(m^*)}{s_P(m^*+1)}}
 \right) \iff
 x^2 \ge \dfrac{1}{64} K_{\mathrm{pen}}^2 n^{-2\kappa} s_P(m) \left( 1 - \sqrt{\dfrac{s_P(m^*)}{s_P(m^*+1)}}
 \right)^2 . $$

Thus, 
\begin{align*}
\mathbb{P}\left(\underset{\substack{\beta_m \in \mathcal{B}_{s_P(m) \times Q,\alpha} \\ R_m \in \mathcal{M}_Q([-1,1])}}{\sup} \ Z_m(\beta_m,R_m) \ge \dfrac{u_{m,n}}{2}
+ Z_m(\beta_m',R_m') \right)
&\le 36 \exp{\left\{ -\dfrac{K^2_{\mathrm{pen}} n^{-2\kappa+1} s_P(m) \left( 1 - \sqrt{\dfrac{s_P(m^*)}{s_P(m^*+1)}} \right)^2}{9216 K^2 \left[K_\text{neighb} K_{\mathcal{Y}} Q^{\frac{3}{2}} + \exp{(K_{\mathcal{X}})} \alpha \right]^2} \right\}}    \\
&= 36 \exp{\left[ -K_1 s_P(m) n^{-2\kappa+1} \right]},
\end{align*}
where $K_1 = \dfrac{K_{\mathrm{pen}}^2 }{9216K^2 \left[ K_\text{neighb} K_{\mathcal{Y}} Q^{\frac{3}{2}} + \exp{(K_{\mathcal{X}})} \alpha \right]^2} \left( 1 - \sqrt{\dfrac{s_P(m^*)}{s_P(m^*+1)}} \right)^2$. \\

\vspace{0.3cm}

\textbullet \ Now we consider $b$:
\begin{align*}
\mathbb{P}\left(Z_m(\beta_m',R_m') > \dfrac{u_{m,n}}{2} \right) & \\
&\hspace{-2.5cm}= \mathbb{P}\left\{ \dfrac{1}{n} \dsum_{i=1}^n \left[ ||Y_i-(W\bm{Y}R_m')_{i,\mydot}-S^m(X_i)\beta_m' ||^2 - \mathbb{E}\left(||\mathcal{Y}_i - (W\bm{\mathcal{Y}}R_m')_{i,\mydot} - S^m(\mathcal{X}_i)\beta_m'||^2 \right) \right] > \dfrac{u_{m,n}}{2} \right\} . 
\end{align*}

Let $B_i = ||Y_i - (W\bm{Y}R_m')_{i,\mydot} - S^m(X_i)\beta_m'||^2$. Then $B_i \ge 0$ and
\begin{align*}
B_i &\le \Big(||Y_i||+||W_{i,\mydot} \bm{Y} || \ ||R_m'|| + ||S^m(X_i)|| \ ||\beta_m'|| \Big)^2 \\
&\le \left[K_{\mathcal{Y}} + \sqrt{Q} K_\text{neighb} K_{\mathcal{Y}} ||R_m'|| + \exp{(K_{\mathcal{X}})} ||\beta_m'|| \right]^2 = \left[K_{\mathcal{Y}}(1 + \sqrt{Q} K_\text{neighb} ||R_m'||) + \exp{(K_{\mathcal{X}})} ||\beta_m'|| \right]^2.
\end{align*}

Now we apply Hoeffding's inequality (Lemma \ref{lemma:Hoeffdingsineq}):
\begin{align*}
\mathbb{P}\left(Z_m(\beta_m',R_m') > \dfrac{u_{m,n}}{2} \right) &=
\mathbb{P}\left\{ \dfrac{1}{n}\dsum_{i=1}^n [B_i-\mathbb{E}(B_i)] > \dfrac{u_{m,n}}{2} \right\} \\
&=
\mathbb{P}\left\{\dsum_{i=1}^n [B_i-\mathbb{E}(B_i)] > \dfrac{n u_{m,n}}{2}\right\} \\
&\le \exp{\left\{ - \dfrac{2 n^2 u_{m,n}^2}{4 n \left[ K_{\mathcal{Y}}(1 + \sqrt{Q} K_\text{neighb}  ||R_m'||) + \exp{(K_{\mathcal{X}})} ||\beta_m'|| \right]^4} \right\}} \\
&= \exp{\left\{ - \dfrac{ n u_{m,n}^2}{2  \left[ K_{\mathcal{Y}}(1 + \sqrt{Q} K_\text{neighb}  ||R_m'||) + \exp{(K_{\mathcal{X}})} ||\beta_m'|| \right]^4} \right\}} \\
&= \exp{\left\{ - \dfrac{ n K_{\mathrm{pen}}^2 n^{-2\kappa} \left(\sqrt{s_P(m)} - \sqrt{s_P(m^*)}\right)^2}{8  \left[ K_{\mathcal{Y}}(1 + \sqrt{Q} K_\text{neighb}  ||R_m'||) + \exp{(K_{\mathcal{X}})} ||\beta_m'|| \right]^4} \right\}} \\
&= \exp{\left\{ - \dfrac{K_{\mathrm{pen}}^2 n^{1-2\kappa} s_P(m) \left(1 - \sqrt{\dfrac{s_P(m^*)}{s_P(m)}}\right)^2}{8  \left[ K_{\mathcal{Y}}(1 + \sqrt{Q} K_\text{neighb}  ||R_m'||) + \exp{(K_{\mathcal{X}})} ||\beta_m'|| \right]^4} \right\}} \\
&\le \exp{\left\{ - \dfrac{K_{\mathrm{pen}}^2 n^{1-2\kappa} s_P(m) \left(1 - \sqrt{\dfrac{s_P(m^*)}{s_P(m^*+1)}}\right)^2}{8  \left[ K_{\mathcal{Y}}(1 + \sqrt{Q} K_\text{neighb}  ||R_m'||) + \exp{(K_{\mathcal{X}})} ||\beta_m'|| \right]^4} \right\}} \\
&= \exp{\left[-K_{2,n} n^{1-2\kappa} s_P(m) \right]},
\end{align*}
with $K_{2,n} = \dfrac{K_{\mathrm{pen}}^2}{8 \left[ K_{\mathcal{Y}} (1 + \sqrt{Q} K_\text{neighb}  ||R_m'|| ) + \exp{(K_{\mathcal{X}})} ||\beta_m'|| \right]^4 } \left(1-\sqrt{\dfrac{s_P(m^*)}{s_P(m^*+1)}}\right)^2$. \\

Then,
$$ \mathbb{P}\left(\underset{\substack{\beta_m \in \mathcal{B}_{s_P(m) \times Q,\alpha} \\ R_m \in \mathcal{M}_Q([-1,1]) } }{\sup} \ Z_m(\beta_m,R_m) \ge u_{m,n} \right) \le 36 \exp{\left[ -K_1 s_P(m) n^{-2\kappa+1} \right]} + \exp{\left[-K_{2,n} n^{1-2\kappa} s_P(m)\right]}.$$
With $K_{3,n} = \min(K_1,K_{2,n})$, we get
$$ \mathbb{P}\left(\underset{\substack{\beta_m \in \mathcal{B}_{s_P(m) \times Q,\alpha} \\ R_m \in \mathcal{M}_Q([-1,1]) }}{\sup} \ Z_m(\beta_m,R_m) \ge u_{m,n} \right) \le 37 \exp{\left[ -K_{3,n} s_P(m) n^{-2\kappa+1} \right]}.$$

Similarly,
$$ \mathbb{P}\left(\underset{\substack{\beta_m \in \mathcal{B}_{s_P(m) \times Q,\alpha} \\ R_m \in \mathcal{M}_Q([-1,1])}}{\sup} \ -Z_m(\beta_m,R_m) \ge u_{m,n} \right) \le 37 \exp{\left[ -K_{3,n} s_P(m) n^{-2\kappa+1} \right]}.$$

Thus, 
$$ \mathbb{P}\left(\underset{\substack{\beta_m \in \mathcal{B}_{s_P(m) \times Q,\alpha} \\ R_m \in \mathcal{M}_Q([-1,1])}}{\sup} \ |Z_m(\beta_m,R_m)| \ge u_{m,n} \right) \le 74 \exp{\left[ -K_{3,n} s_P(m) n^{-2\kappa+1} \right]}.$$

And
\begin{align*}
\mathbb{P}(\widehat{m}=m^* ) &\le 74 \exp{\left[ -K_{3,n} s_P(m) n^{-2\kappa+1} \right]}.
\end{align*}

To optimize the upper bound, we maximize $K_{3_n}$ and so $K_{2,n}$ according to $\beta_m'$ and $R_m'$:
$$ K_{2,n} = \dfrac{K_{\mathrm{pen}}^2}{8 \left[ K_{\mathcal{Y}}(1 + \sqrt{Q} K_\text{neighb} ||R_m'||) + \exp{(K_{\mathcal{X}})} ||\beta_m'|| \right]^4 } \left( 1-\sqrt{\dfrac{s_P(m^*)}{s_P(m^*+1)}} \right)^2 $$ is maximum when $||\beta_m'||=||R_m'||=0$. Then we have $$ K_{2,n} = K_2 = \dfrac{K_{\mathrm{pen}}^2}{8 K_{\mathcal{Y}}^4 } \left( 1-\sqrt{\dfrac{s_P(m^*)}{s_P(m^*+1)}} \right)^2. $$

Finally, since $$K_1 = \dfrac{K_{\mathrm{pen}}^2}{9216 K^2 \left[ K_\text{neighb} K_{\mathcal{Y}} Q^{\frac{3}{2}} + \exp{(K_{\mathcal{X}})} \alpha \right]^2 } \left( 1-\sqrt{\dfrac{s_P(m^*)}{s_P(m^*+1)}} \right)^2, $$ we have 
$$ K_{3,n} = K_3 = \left( 1-\sqrt{\dfrac{s_P(m^*)}{s_P(m^*+1)}} \right)^2 \dfrac{K^2_{\mathrm{pen}}}{8} \min{\left( \dfrac{1}{K_{\mathcal{Y}}^4} , \dfrac{1}{1152K^2 \left[ K_\text{neighb} K_{\mathcal{Y}} Q^{\frac{3}{2}} + \exp{(K_{\mathcal{X}})}\alpha \right]^2} \right)}. $$

\end{proof}

\begin{proposition} \label{prop3} \phantom{linebreak} \\
For any $\delta > 0$, $m\in \mathbb N$, let $n_2$ be the smaller integer such that
$$
n_2 \ge \dfrac{\left[432K\left(\alpha \exp{(K_\mathcal{X})} \sqrt{s_P(m)\pi} + K_\text{neighb} K_\mathcal{Y} Q^{\frac{5}{2}}\sqrt{\pi}\right) \right]^2}{\delta^2}.
$$
Then for $n \ge n_2$,
$$
\mathbb{P}\left( |\widehat{L}(m) - L(m)| \ge \delta \right) \le 74 \exp{\left( -n\delta^2 K_4 \right)},
$$
where 
$$K_4 = \min{\left( \dfrac{1}{2304K^2 [K_\text{neighb} K_{\mathcal{Y}} Q^{\frac{3}{2}} +  \exp{(K_{\mathcal{X}})} \alpha]^2}, \dfrac{1}{2K_\mathcal{Y}^4} \right)}.
$$

\end{proposition}

\vspace{0.3cm}

\begin{proof}

We deduce from Lemma \ref{lemma1fermanian}  that
\begin{align*}
\mathbb{P} \left( |\widehat{L}(m) - L(m)| \ge \delta \right) &\le \mathbb{P} \left( \underset{\substack{\beta_m \in \mathcal{B}_{s_P(m) \times Q,\alpha} \\ R_m \in \mathcal{M}_Q([-1,1])}}{\sup} | \widehat{\mathcal{R}}_m(\beta_m, R_m) - \mathcal{R}_m(\beta_m, R_m) | \ge \delta \right) \\
&\le \mathbb{P} \left( \underset{\substack{\beta_m \in \mathcal{B}_{s_P(m) \times Q,\alpha} \\ R_m \in \mathcal{M}_Q([-1,1])}}{\sup} |Z_m(\beta_m, R_m)| \ge \delta \right) \\
&\le \mathbb{P} \left( \underset{\substack{\beta_m \in \mathcal{B}_{s_P(m) \times Q,\alpha} \\ R_m \in \mathcal{M}_Q([-1,1])}}{\sup} Z_m(\beta_m, R_m) \ge \delta \right) + \mathbb{P} \left( \underset{\substack{\beta_m \in \mathcal{B}_{s_P(m) \times Q,\alpha} \\ R_m \in \mathcal{M}_Q([-1,1])}}{\sup} -Z_m(\beta_m, R_m) \ge \delta \right) .
\end{align*}

Let's fix $\beta'_{m} \in \mathcal{B}_{s_P(m) \times Q,\alpha}, R'_m \in \mathcal{M}_Q([-1,1])$, then

\begin{align*}
\mathbb{P}\left( \underset{\substack{\beta_m \in \mathcal{B}_{s_P(m) \times Q, \alpha} \\ R_m \in \mathcal{M}_Q([-1,1])}}{\sup} Z_m(\beta_m, R_m) \ge \delta\right) &= \mathbb{P}\left( \underset{\substack{\beta_m \in \mathcal{B}_{s_P(m) \times Q, \alpha} \\ R_m \in \mathcal{M}_Q([-1,1])}}{\sup} Z_m(\beta_m, R_m) \ge \delta, Z_m(\beta_m', R_m') \le \dfrac{\delta}{2} \right) \\
&+ \mathbb{P}\left( \underset{\substack{\beta_m \in \mathcal{B}_{s_P(m) \times Q, \alpha} \\ R_m \in \mathcal{M}_Q([-1,1])}}{\sup} Z_m(\beta_m, R_m) \ge \delta, Z_m(\beta_m', R_m') > \dfrac{\delta}{2} \right) \\
&\le 
\underbrace{\mathbb{P}\left( \underset{\substack{\beta_m \in \mathcal{B}_{s_P(m) \times Q, \alpha} \\ R_m \in \mathcal{M}_Q([-1,1])}}{\sup} Z_m(\beta_m, R_m) \ge Z_m(\beta_m', R_m') + \dfrac{\delta}{2} \right)}_a \\
&+ \underbrace{\mathbb{P}\left(Z_m(\beta'_{m}, R'_m) > \dfrac{\delta}{2}\right)}_b
\end{align*}

Denote $x = \dfrac{\delta}{2} - 108 K \left( \alpha \dfrac{\exp{(K_\mathcal{X})}}{\sqrt n} \sqrt{s_P(m)\pi} + \dfrac{K_\text{neighb} K_\mathcal{Y} Q^{\frac{5}{2}}\sqrt{\pi}}{\sqrt{n}} \right)$, then for $n\ge n_2$,
$$
x \ge \dfrac{\delta}{4} > 0
$$
and we get by applying Proposition \ref{prop1} on $a$:

$$ \mathbb{P}\left( \underset{\substack{\beta_m \in \mathcal{B}_{s_P(m) \times Q, \alpha} \\ R_m \in \mathcal{M}_Q([-1,1])}}{\sup} Z_m(\beta_m, R_m) \ge Z_m(\beta_m', R_m') + \dfrac{\delta}{2} \right) \le 36 \exp{\left\{ -\dfrac{x^2 n}{144K^2 [K_\text{neighb} K_{\mathcal{Y}} Q^{\frac{3}{2}} +  \exp{(K_{\mathcal{X}})} \alpha]^2}\right\}} $$

Now, from Hoeffding's inequality (Lemma \ref{lemma:Hoeffdingsineq}) and by using the same $B_i$ random variables as in proof of Proposition \ref{prop2}, we have
$$ \mathbb{P}\left(Z_m(\beta'_{m}, R'_m) > \dfrac{\delta}{2}\right) \le \exp{\left\{ - \dfrac{2 n^2 \delta^2}{4 n \left[ K_{\mathcal{Y}}(1 + \sqrt{Q} K_\text{neighb}  ||R_m'||) + \exp{(K_{\mathcal{X}})} ||\beta_m'|| \right]^4}\right\}} $$

Thus
\begin{align*}
\mathbb{P}\left( \underset{\substack{\beta_m \in \mathcal{B}_{s_P(m) \times Q, \alpha} \\ R_m \in \mathcal{M}_Q([-1,1])}}{\sup} Z_m(\beta_m, R_m) \ge \delta\right) &\le 36 \exp{\left\{ -\dfrac{x^2 n}{144K^2 [K_\text{neighb} K_{\mathcal{Y}} Q^{\frac{3}{2}} +  \exp{(K_{\mathcal{X}})} \alpha]^2}\right\}} \\
&+ \exp{\left\{ - \dfrac{2 n^2 \delta^2}{4 n \left[ K_{\mathcal{Y}}(1 + \sqrt{Q} K_\text{neighb}  ||R_m'||) + \exp{(K_{\mathcal{X}})} ||\beta_m'|| \right]^4}\right\}} \\
&\le 36 \exp{\left\{-\dfrac{\delta^2 n}{2304 K^2 [K_\text{neighb} K_{\mathcal{Y}} Q^{\frac{3}{2}} +  \exp{(K_{\mathcal{X}})} \alpha]^2}\right\}} \\
&+ \exp{\left\{ - \dfrac{n \delta^2}{2 \left[ K_{\mathcal{Y}}(1 + \sqrt{Q} K_\text{neighb}  ||R_m'||) + \exp{(K_{\mathcal{X}})} ||\beta_m'|| \right]^4}\right\}}.
\end{align*}

We minimize the upper bound according to $\beta'_{m}$ and $R'_m$ which results in $||\beta'_{m}|| = ||R'_m|| = 0$ and gives us
$$
\mathbb{P}\left( \underset{\substack{\beta_m \in \mathcal{B}_{s_P(m) \times Q, \alpha} \\ R_m \in \mathcal{M}_Q([-1,1])}}{\sup} Z_m(\beta_m, R_m) \ge \delta\right) \le 37 \exp{\left\{ -n\delta^2 K_4 \right\}},
$$
where $K_4 = \min{\left( \dfrac{1}{2304K^2 [K_\text{neighb} K_{\mathcal{Y}} Q^{\frac{3}{2}} +  \exp{(K_{\mathcal{X}})} \alpha]^2}, \dfrac{1}{2K_\mathcal{Y}^4} \right)} $. \\

We prove the same way that 
$$
\mathbb{P}\left( \underset{\substack{\beta_m \in \mathcal{B}_{s_P(m) \times Q, \alpha} \\ R_m \in \mathcal{M}_Q([-1,1])}}{\sup} -Z_m(\beta_m, R_m) \ge \delta \right) \le 37 \exp{\left( -n\delta^2 K_4\right)}.
$$

This finally gives us 
$$
\mathbb{P}\left(|\widehat{L}(m) - L(m)| \ge \delta \right) \le 74 \exp{\left( -n\delta^2 K_4\right)}.
$$

\end{proof}

\begin{proposition} \label{prop4} \phantom{linebreak} \\

Let $0< \kappa < \dfrac{1}{2}$ and $\mathrm{pen}_n(m) = K_{\mathrm{pen}} n^{-\kappa} \sqrt{s_P(m)}$. Let $n_3$ be the smallest integer satisfying
$$
n_3 \geq \left(\dfrac{1728K\left(\alpha \exp{(K_\mathcal{X})} \sqrt{s_P(m)\pi} + K_\text{neighb} K_\mathcal{Y} Q^{\frac{5}{2}}\sqrt{\pi}\right) +2K_{\mathrm{pen}} \sqrt{s_P(m^*)}}{L(m^*-1)-\sigma^2}\right)^{\frac{1}{\kappa}}, 
$$
with $\sigma^2 = \text{Tr}(\Sigma)$.
Then for any $m<m^*$, $n\geq n_3$,

$$
\mathbb{P}(\widehat{m} = m) \le 148\exp{\left\{ -n\dfrac{K_4}{4} [L(m) - L(m^*) - \mathrm{pen}_n(m^*) + \mathrm{pen}_n(m)]^2 \right\}}.
$$

\end{proposition}

\begin{proof}
Since $\widehat{m} = \min (\underset{m \in \mathbb{N}^*}{\arg \min} (\widehat{L}(m) + \mathrm{pen}_n(m)))$, for $m<m^*$,
\begin{align*}
\mathbb{P}(\widehat{m} = m) &\le \mathbb{P}(\widehat{L}(m) - \widehat{L}(m^*) \le \mathrm{pen}_n(m^*) - \mathrm{pen}_n(m))\\
&= \mathbb{P}(\widehat{L}(m) - \widehat{L}(m^*) + L(m^*) - L(m) \le \mathrm{pen}_n(m^*) - \mathrm{pen}_n(m) + L(m^*) - L(m))\\
&= \mathbb{P}(\widehat{L}(m^*) - L(m^*) + L(m) - \widehat{L}(m) \ge L(m) - L(m^*) - \mathrm{pen}_n(m^*) + \mathrm{pen}_n(m))\\
&\le \mathbb{P}\left( \widehat{L}(m^*) - L(m^*) \ge \dfrac{1}{2}(L(m) - L(m^*) - \mathrm{pen}_n(m^*) + \mathrm{pen}_n(m)) \right) \\
&+ \mathbb{P}\left(L(m) - \widehat{L}(m) \ge \dfrac{1}{2}(L(m) - L(m^*) - \mathrm{pen}_n(m^*) + \mathrm{pen}_n(m))\right) \\
&\le \mathbb{P}\left( |\widehat{L}(m^*) - L(m^*)| \ge \dfrac{1}{2}(L(m) - L(m^*) - \mathrm{pen}_n(m^*) + \mathrm{pen}_n(m)) \right) \\
&+ \mathbb{P}\left(|L(m) - \widehat{L}(m)| \ge \dfrac{1}{2}(L(m) - L(m^*) - \mathrm{pen}_n(m^*) + \mathrm{pen}_n(m))\right).
\end{align*}

Let's denote $x = \dfrac{1}{2}(L(m) - L(m^*) - \mathrm{pen}_n(m^*) + \mathrm{pen}_n(m))$. The function $L(m)$ is decreasing and bounded by $\sigma^2$, thus for $m < m^*$ one has that  $ L(m) \geq L(m^* - 1)$. Moreover,  $L(m^*) = \sigma^2$ which gives us, $L(m) - L(m^*)  \geq L(m^*-1) - \sigma^2$. Furthermore, $\mathrm{pen}_n(m)$ is stricly increasing, therefore
$$
2x \ge L(m^* - 1) - \sigma^2 - K_{\mathrm{pen}} n^{-\kappa} \sqrt{s_P(m^*)}.
$$
We can deduce from this inequality that if $n \ge \left( \dfrac{2K_{\mathrm{pen}} \sqrt{s_P(m^*)}}{L(m^*-1)-\sigma^2} \right)^{\frac{1}{\kappa}}$, then
$$
x \ge \dfrac{1}{4}(L(m^*-1)-\sigma^2) > 0.
$$
Since $x > 0$, we are able to apply Proposition \ref{prop3} with $\delta = x$ if $n$ satisfies 

\begin{align*}
n &\ge \dfrac{\left[432K\left(\alpha \exp{(K_\mathcal{X})} \sqrt{s_P(m)\pi} + K_\text{neighb} K_\mathcal{Y} Q^{\frac{5}{2}}\sqrt{\pi}\right) \right]^2}{x^2}\\
&\ge \dfrac{\left[1728K\left(\alpha \exp{(K_\mathcal{X})} \sqrt{s_P(m)\pi} + K_\text{neighb} K_\mathcal{Y} Q^{\frac{5}{2}}\sqrt{\pi}\right) \right]^2}{[L(m^*-1)-\sigma^2]^2}.
\end{align*}
If this bound is below $1$, then this condition is trivially satisfied and we denote $n_3 = \left\lceil \left( \dfrac{2K_{\mathrm{pen}} \sqrt{s_P(m^*)}}{L(m^*-1)-\sigma^2} \right)^{\frac{1}{\kappa}} \right\rceil$. Otherwise, for $0 < \kappa < \dfrac{1}{2}$, by combining the two bounds on $n$, we denote $$
n_3 = \left\lceil \max\left( \dfrac{1728K\left(\alpha \exp{(K_\mathcal{X})} \sqrt{s_P(m)\pi} + K_\text{neighb} K_\mathcal{Y} Q^{\frac{5}{2}}\sqrt{\pi}\right) }{L(m^*-1)-\sigma^2} ,   \dfrac{2K_{\mathrm{pen}} \sqrt{s_P(m^*)}}{L(m^*-1)-\sigma^2}\right)^{\frac{1}{\kappa}} \right\rceil,
$$ and for $n \ge n_3$ we apply Proposition \ref{prop3}, with $\delta = x$:

\begin{align*}
\mathbb{P}(\widehat{m} = m) &\le 148 \exp{(-n x^2 K_4)} \\
&= 148 \exp{\left\{-n \dfrac{K_4}{4} [L(m) - L(m^*) - \mathrm{pen}_n(m^*) + \mathrm{pen}_n(m)]^2 \right\}}
\end{align*}

\end{proof}

\textbf{Recall of Theorem \ref{theorem1}.} \\

Let $0<\kappa < \dfrac{1}{2}$, $\mathrm{pen}_n(m) = K_{\mathrm{pen}} n^{-\kappa} \sqrt{s_P(m)}$ and $n \geq \max (n_1,n_3)$, then

$$
\mathbb{P}(\widehat{m} \neq m^*) \le 148 m^* \exp{\left\{ -n\dfrac{K_4}{16} \left[ L(m^*-1)-\sigma^2 \right]^2\right\}} + 74\sum_{m > m^*}\exp{\left\{ - K_3 s_P(m) n^{-2\kappa+1} \right\}}.
$$

\begin{proof}
$$
\mathbb{P}(\widehat{m} \neq m^*] = \mathbb{P}(\widehat{m} < m^*) + \mathbb{P}(\widehat{m} > m^*) \le \sum_{m < m^*} \mathbb{P}(\widehat{m} = m) + \sum_{m > m^*} \mathbb{P}(\widehat{m} = m).
$$

One can deduce from Proposition \ref{prop2} that

$$
\sum_{m > m^*} \mathbb{P}(\widehat{m} = m) \le 74\sum_{m > m^*}\exp{\left[ - K_3 s_P(m) n^{-2\kappa+1} \right]}.
$$
The other sum can be handle through Proposition \ref{prop4}, indeed as we proved in the proof of Proposition \ref{prop4} that 
$$
L(m) - L(m^*) - \mathrm{pen}_n(m^*) + \mathrm{pen}_n(m) \geq \dfrac{1}{2}(L(m^*-1)-\sigma^2),
$$
as long as $n\geq n_3$. Thus, one has that
:

\begin{align*}
\sum_{m < m^*} \mathbb{P}(\widehat{m} = m) &\le 148\sum_{m < m^*} \exp{\left\{ -n\dfrac{K_4}{4} [L(m) - L(m^*) - \mathrm{pen}_n(m^*) + \mathrm{pen}_n(m)]^2 \right\}}\\
&\le 148 m^*\exp{\left\{ -n\dfrac{K_4}{16}[L(m^*-1)-\sigma^2]^2\right\}}.
\end{align*}

We conclude the proof of this theorem by combining these two bounds:

$$
\mathbb{P}(\widehat{m} \neq m^*) \le 148 m^* \exp{\left\{ -n\dfrac{K_4}{16}[L(m^*-1)-\sigma^2]^2\right\}} + 74\sum_{m > m^*}\exp{\left\{ - K_3 s_P(m) n^{-2\kappa+1} \right\}}.
$$

\end{proof}

\section{Results of the simulation study: Estimation of $R$} \label{appendix:ressim}

Figures \ref{fig:estimRw}, \ref{fig:estimRm} and \ref{fig:estimRh} present the estimation of $R$ when considering weak, moderate and high spatial effects, respectively.

\begin{figure}
\centering
\includegraphics[width=\textwidth]{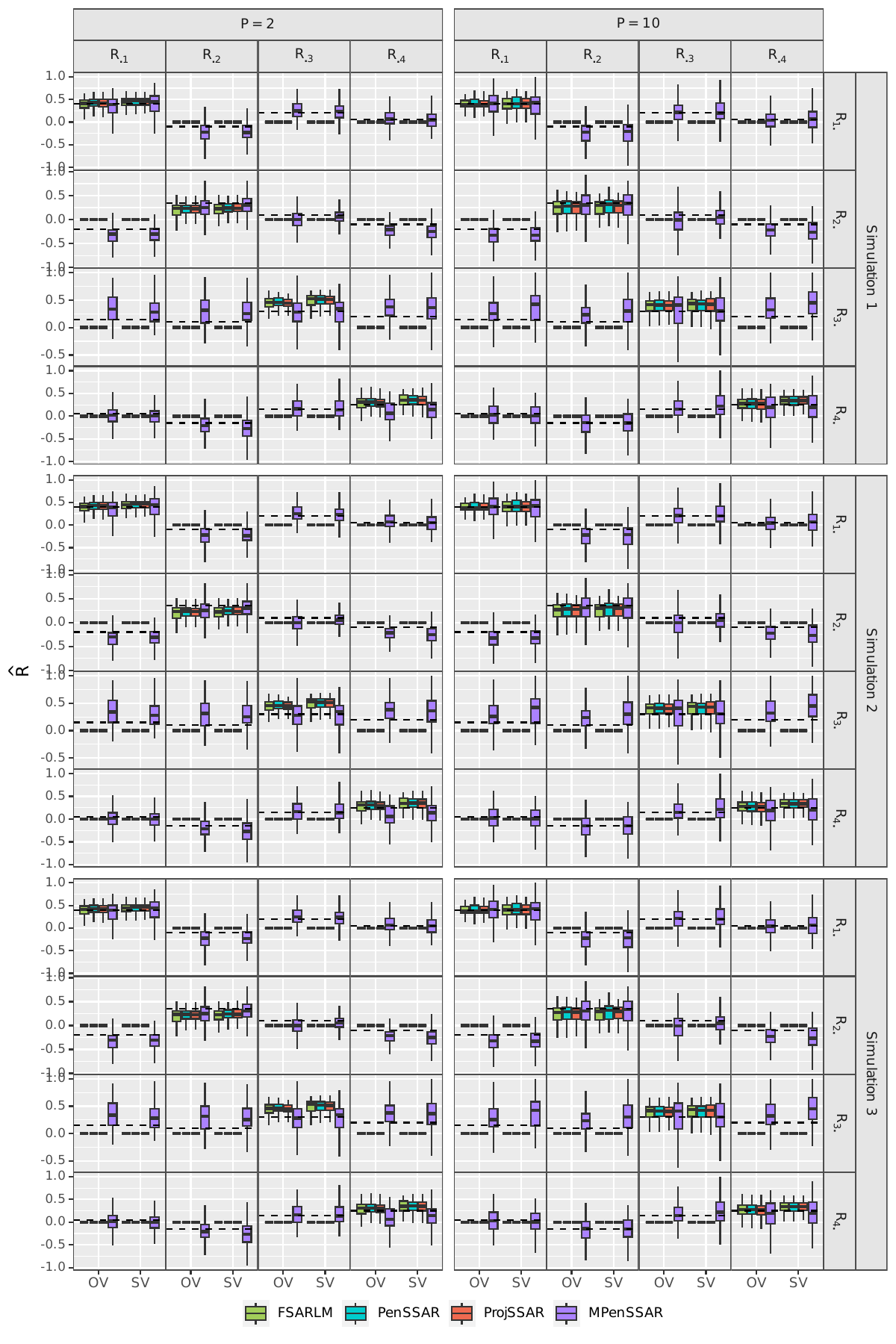}
\caption{Estimation of $R$ with the FSARLM, the PenSSAR, the ProjSSAR and the MPenSSAR using ordinary (OV) and spatial (SV) validation, when $R = R_w$. The dotted line corresponds to the true value of the coefficient.}
\label{fig:estimRw}
\end{figure}

\begin{figure}
\centering
\includegraphics[width=\textwidth]{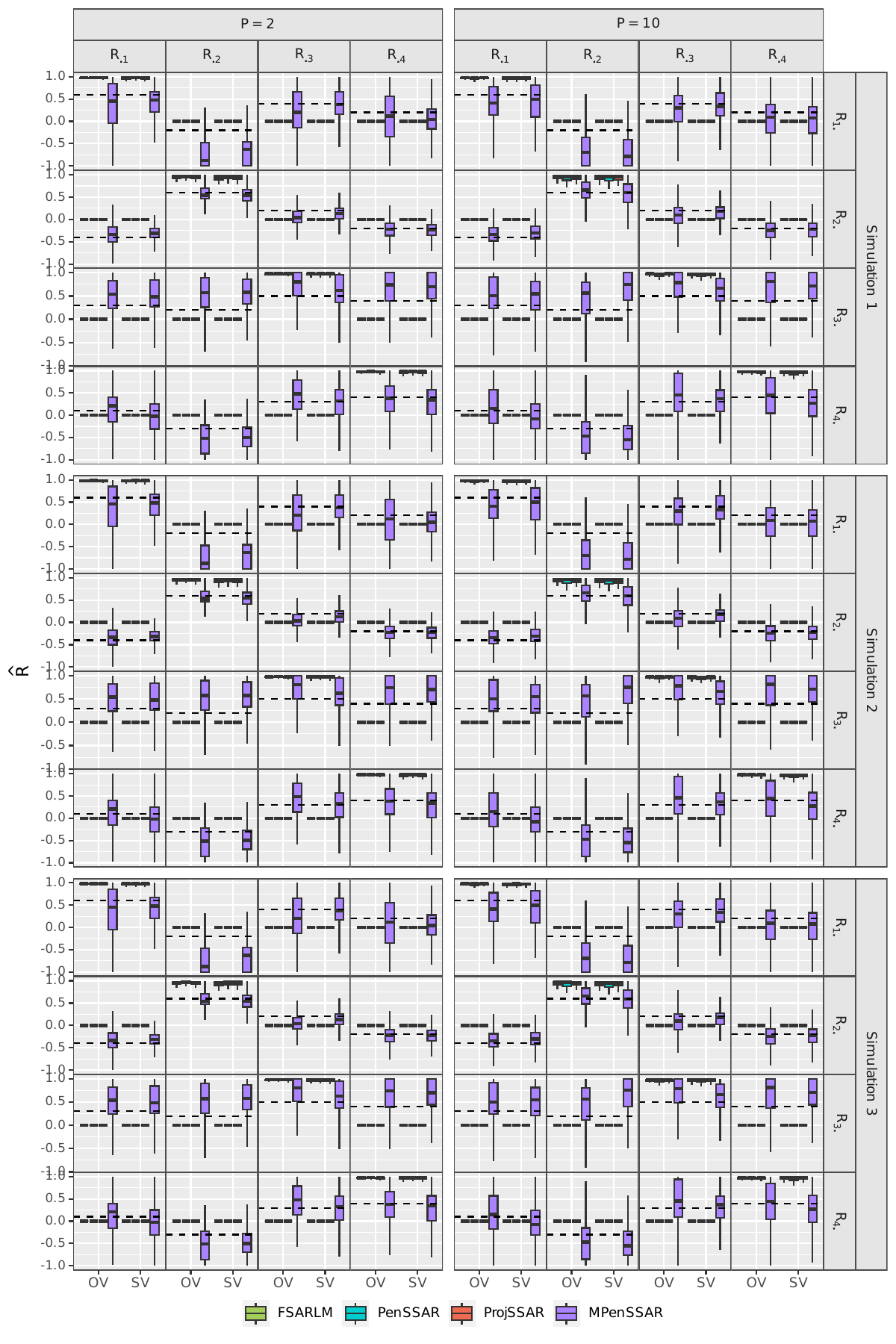}
\caption{Estimation of $R$ with the FSARLM, the PenSSAR, the ProjSSAR and the MPenSSAR using ordinary (OV) and spatial (SV) validation, when $R = R_\text{mod}$. The dotted line corresponds to the true value of the coefficient.}
\label{fig:estimRm}
\end{figure}

\begin{figure}
\centering
\includegraphics[width=\textwidth]{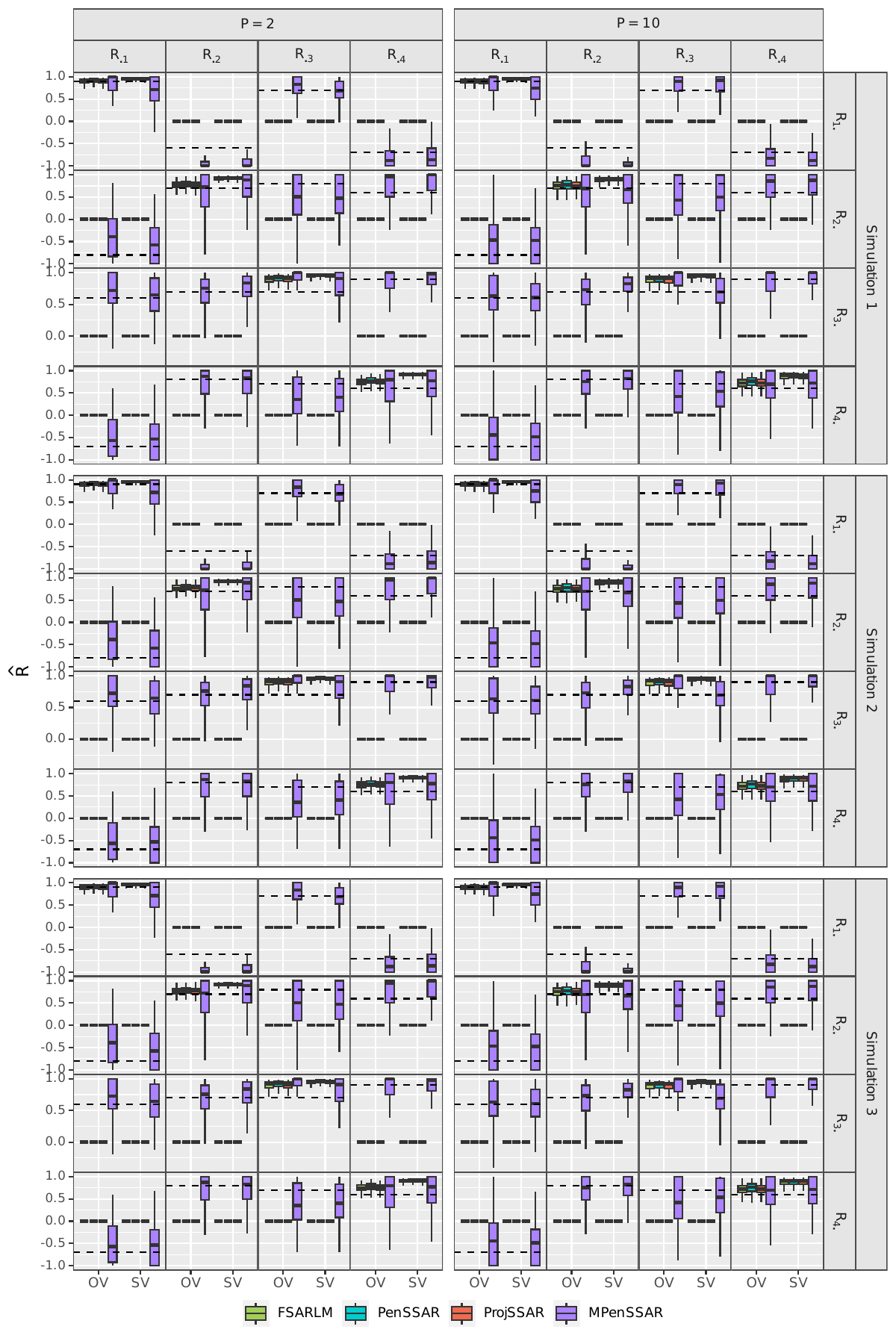}
\caption{Estimation of $R$ with the FSARLM, the PenSSAR, the ProjSSAR and the MPenSSAR using ordinary (OV) and spatial (SV) validation, when $R = R_h$. The dotted line corresponds to the true value of the coefficient.}
\label{fig:estimRh}
\end{figure}

\end{document}